\documentclass{amsart}
\usepackage{amssymb}
\usepackage{braket}
\usepackage{mathrsfs}
\usepackage{ifthen}
\usepackage{graphicx}

\usepackage{comment} 
\usepackage{mleftright}
\usepackage{hyperref}
\usepackage{cleveref}
 \usepackage[all]{xy}
 \usepackage{amscd}
 \usepackage{amsrefs}
 \usepackage{color}
 \usepackage{enumitem}
\newlist{steps}{enumerate}{1}
\setlist[steps, 1]{label = Step \arabic*:}
\usepackage[fontsize=10]{scrextend}

\makeatletter
\DeclareRobustCommand\widecheck[1]{{\mathpalette\@widecheck{#1}}}
\def\@widecheck#1#2{%
   \setbox\z@\hbox{\m@th$#1#2$}%
   \setbox\tw@\hbox{\m@th$#1%
      \widehat{%
         \vrule\@width\z@\@height\ht\z@
         \vrule\@height\z@\@width\wd\z@}$}%
   \dp\tw@-\ht\z@
   \@tempdima\ht\z@ \advance\@tempdima2\ht\tw@ \divide\@tempdima\thr@@
   \setbox\tw@\hbox{%
      \raise\@tempdima\hbox{\scalebox{1}[-1]{\lower\@tempdima\box\tw@}}}%
   {\ooalign{\box\tw@ \cr \box\z@}}}
\makeatother

\theoremstyle{plain}
\newtheorem{thm}{Theorem}[section]
\crefname{thm}{Theorem}{Theorems}
\Crefname{thm}{Theorem}{Theorems}
\newtheorem{prop}[thm]{Proposition}
\crefname{prop}{Proposition}{Propositions}
\Crefname{prop}{Proposition}{Propositions}
\newtheorem{lem}[thm]{Lemma}
\crefname{lem}{Lemma}{Lemmas}
\Crefname{lem}{Lemma}{Lemmas}
\newtheorem{cor}[thm]{Corollary}
\crefname{cor}{Corollary}{Corollaries}
\Crefname{cor}{Corollary}{Corollaries}

\crefname{claim}{Claim}{Claims}
\Crefname{claim}{Claim}{Claims}

\crefname{property}{Property}{Properties}
\Crefname{property}{Property}{Properties}

\crefname{problem}{Problem}{Problems}
\Crefname{problem}{Problem}{Problems}
\newtheorem{ques}[thm]{Question}
\crefname{ques}{Question}{Questions}
\Crefname{ques}{Question}{Questions}

\theoremstyle{definition}
\newtheorem{defn}[thm]{Definition}
\crefname{defn}{Definition}{Definitions}
\Crefname{defn}{Definition}{Definitions}

\crefname{notation}{Notation}{Notations}
\Crefname{notation}{Notation}{Notations}

\crefname{convention}{Convention}{Conventions}
\Crefname{convention}{Convention}{Conventions}

\crefname{cond}{Condition}{Conditions}
\Crefname{cond}{Condition}{Conditions}

\crefname{assum}{Assumption}{Assumptions}
\Crefname{assum}{Assumption}{Assumptions}
\newtheorem{conj}[thm]{Conjecture}
\crefname{conj}{Conjecture}{Conjectures}
\Crefname{conj}{Conjecture}{Conjectures}

\theoremstyle{remark}
\newtheorem{rem}[thm]{Remark}
\crefname{rem}{Remark}{Remarks}
\Crefname{rem}{Remark}{Remarks}
\newtheorem{ex}[thm]{Example}
\crefname{ex}{Example}{Examples}
\Crefname{ex}{Example}{Examples}

\crefname{section}{Section}{Sections}
\Crefname{section}{Section}{Sections}
\crefname{subsection}{Subsection}{Subsections}
\Crefname{subsection}{Subsection}{Subsections}
\crefname{figure}{Figure}{Figures}
\Crefname{figure}{Figure}{Figures}

\newtheorem*{acknowledgement}{Acknowledgement}

\newcommand{\Z}{\mathbb{Z}}

\newcommand{\Q}{\mathbb{Q}}

\newcommand{\Ker}{\mathop{\mathrm{Ker}}\nolimits}
\newcommand{\Coker}{\mathop{\mathrm{Coker}}\nolimits}
\newcommand{\im}{\operatorname{Im}}
\newcommand{\rank}{\mathop{\mathrm{rank}}\nolimits}
\newcommand{\Hom}{\mathop{\mathrm{Hom}}\nolimits}

\newcommand{\id}{\mathrm{id}}
\newcommand{\ind}{\mathop{\mathrm{ind}}\nolimits}

\newcommand{\G}{\mathcal G}

\newcommand{\C}{\mathbb{C}}
\newcommand{\s}{\mathfrak{s}}
\newcommand{\wh}{\widehat}

\newcommand{\pr}{\text{pr}}
\newcommand{\al}{\alpha}

\def\om{\omega}
\def\Om{\Omega}

\newcommand{\R}{\mathbb R}

\newcommand{\ctext}[1]{\raise0.2ex\hbox{\textcircled{\scriptsize{#1}}}}

\def\ker{\operatorname{Ker}}

\def\dim{\operatorname{dim}}
\def\rank{\operatorname{rank}}

\def\Hom{\operatorname{Hom}}

\def\Con{\mathcal{C}}

\def\id{\operatorname{Id}}

\def\ind{\operatorname{ind}}

\newcommand{\mbar}[1]{{\ooalign{\hfil#1\hfil\crcr\raise.167ex\hbox{--}}}}

\def\cV{\mathcal{V}}

\def\cF{\mathcal{F}}

\def\cU{\mathcal{U}}

\def\wt{\widetilde}
\def\H{\mathbb{H}}

\title{Seiberg-Witten Floer homotopy contact invariant}

\author{Nobuo Iida}
\address{Graduate School of Mathematical Sciences, the University of Tokyo, 3-8-1 Komaba, Meguro, Tokyo 153-8914, Japan}
\email{iida@ms.u-tokyo.ac.jp}

\author{Masaki Taniguchi}
\address{2-1 Hirosawa, Wako, Saitama 351-0198, Japan}
\email{masaki.taniguchi@riken.jp}

\date{}
\begin{document}
\begin{abstract}
We introduce a Floer homotopy version of the contact invariant introduced by Kronheimer-Mrowka-Ozv\'ath-Szab\'o. 
Moreover, we prove a gluing formula relating our invariant with the first author's Bauer-Furuta type invariant, which refines Kronheimer-Mrowka's invariant for 4-manifolds with contact boundary. As an application, we give a constraint for a certain class of symplectic fillings using equivariant KO-cohomology. 
\end{abstract}

\maketitle

\tableofcontents

\section{Introduction} 
\subsection{Main theorems} 
In the past twenty years, the topology of symplectic fillings of contact three-manifolds has been a central topic of research at the intersection of symplectic geometry, gauge theory, and Heegaard Floer theory (\cite{LM97}, \cite{KM97},\cite{OS04},  \cite{OS05}, \cite{KMOS07}, \cite{BS16}). 
One of the origins is the study of Kronheimer-Mrowka (\cite{KM97}). They developed the analysis on $4$-manifolds with cone-like ends and gave an invariant 
\begin{align} \label{KM}
\mathfrak{m}(X, \s_{X, \xi}, \xi ) \in \Z / \{\pm 1\} 
\end{align}
of any $4$-manifold $X$ equipped with a contact structure $\xi$ on its boundary and a compatible $Spin^c$-structure $s_{X, \xi}$. This is a variant of the Seiberg-Witten invariant(\cite{W94}). 

In gauge theory, the framework of Floer homology groups gives a cut-and-paste method of $4$-manifold-invariant. In the Seiberg-Witten side, Kronheimer-Mrowka constructed the monopole Floer homology group with three flavors in \cite{KM07}.  As a relative version of \eqref{KM}, Kronheimer-Mrowka-Ozv\'ath-Szab\'o (\cite{KMOS07}) defined a monopole-Floer-homology-valued invariant of a contact structure $\xi$ on a closed 3-manifold $Y$
\begin{align}\label{KMOS}
\psi(Y, \xi) \in \widecheck{HM}_\bullet(-Y), 
\end{align}
which gives subtle information of contact structures such as fillability or overtwistedness. 
As one of applications of such invariants, it is proved that any strong symplectic filling $(X, \om)$ of any L-space has $b^+(X)=0$(\cite{Li98}, \cite{OS04}, \cite{E20}). 
This result was originally proved by Ozv\'ath-Szab\'o using a Heegaard Floer counterpart of \eqref{KMOS} in \cite{OS04}. F.Lin (\cite{L20}) used $Pin(2)$-monopole Floer homology to give a topological constraint of some indefinite Stein fillings.

In this paper, we follow their methods to obtain topological constraints of fillings. In addition, we use a Floer homotopy theoretic viewpoint. More preciously, we construct a Floer homotopy version of \eqref{KMOS}. In order to explain what we mean by {\it Floer homotopy version}, we review the Seiberg-Witten homotopy type, which is constructed by a method called {\it finite dimensional approximation}.

Originally, Furuta (\cite{F01}) introduced the method of finite dimensional approximation of the Seiberg-Witten map and proved 10/8-theorem for closed spin 4-manifolds. 
Later, Bauer-Furuta (\cite{BF04}, \cite{B04})  used this method to construct a cohomotopy refinement of the Seiberg-Witten invariant called Bauer-Furuta invariant, which is an $S^1$-stable homotopy class of an $S^1$-equivariant map. In \cite{Man03}, as a TQFT like extension of the {\it Bauer-Furuta invariant}, Manolescu constructed the {\it Seiberg-Witten Floer homotopy type} for rational homology 3-spheres and the relative Bauer-Furuta invariant for a certain class of 4-manifolds with boundary. 

We construct a homotopy refinement of \eqref{KMOS}, which is a stable homotopy class of a map whose codomain is the Manolescu's Floer homotopy type: 
\begin{thm} \label{mainthm1}
Let $Y$ be a rational homology $3$-sphere equipped with a contact structure $\xi$.
We denote by $d_3(Y, [\xi])$ the homotopy invariant of 2-plane filed introduced Gompf \cite{Go98} with the convention 
\[
d_3(Y, [\xi]) =\frac{1}{4} ( c_1 (X)^2 - 2 \chi (X) - 3\sigma(X) )
\]
where $X$ is a compact almost complex bounding of $(Y, \xi)$. 

Then we can associate a well-defined homotopy class of a non-equivariant pointed map 
\begin{align}\label{ourinv}
\Psi(Y ,\xi ) : S^0 \to \Sigma^{\frac{1}{2}-d_3(-Y, [\xi]) } SWF (-Y, \s_{\xi}) 
\end{align}
up to suspension and sign, where $\s_{\xi}$ is the $Spin^c$ structure induced from $\xi$.
\end{thm}
Moreover, our invariant $\Psi(Y, \xi)$ can be regarded as a relative version of the first author's Bauer-Furuta type invariant (\cite{I19}) 
\begin{align}\label{Iida1}
\Psi(X, \s_{X, \xi}, \xi) : S^{\langle e(S^+, \Phi_0), [(X, \partial X)]\rangle } \to S^0, 
\end{align}
which refines \eqref{KM}, where 
\[
\langle e(S^+, \Phi_0), [(X, \partial X)]\rangle
\]
 is the relative Euler number of the pair $(S^+, \Phi_0)$ of the spinor bundle and its canonical non-vanishing section. 
The following table provides relations between the invariants explained above. 
\begin{table}[htb] \tiny
  \begin{tabular}{|l|c|c|}\hline
  & Counting   & Finite dimensional approximation  \\ \hline
  closed 4-manifolds  &SW-invariant $\in \Z$& BF-invariant 
    \\ 
     & & $\Psi(X):(\R^m \oplus  \C^n)^+ \to (\R^{m'} \oplus  \C^{n'})^+$ \\ \hline
   4-manifolds with & KM-invariant $\in \Z/ \{\pm 1\}$  &  BF-type invariant  \eqref{Iida1} \\ 
  contact boundary   & &$\Psi(X, \xi): (\R^M)^+ \to (\R^{M'} )^+$  \\ \hline
   closed 3-manifolds &monopole Floer homology group &SW Floer homotopy type  \\ 
    & "${HM}_\bullet(Y)$" & $SWF(Y)$ \\ \hline
  4-manifolds with   & relative SW invariant  & relative BF invariant \\  
 boundary &"$\psi(X) \in {HM}_\bullet(\partial X)$" & $\Psi(X): (\R^m \oplus  \C^n)^+ \to  SWF(\partial X)$ \\ \hline
    contact 3-manifolds & contact invariant  & homotopy contact invariant \\ 
     &$\psi(Y, \xi) \in \widecheck{HM}_\bullet(-Y)$ &
     $\Psi(Y, \xi): (\R^M)^+ \to SWF(-Y)$ \\ \hline
  \end{tabular}
\end{table}

Moreover, we prove a gluing relation between \eqref{Iida1} and 
\eqref{ourinv}. 
Let 
\[
\eta : SWF (Y, \s_{\xi}) \wedge SWF (-Y, \s_{\xi})  \to S^0
\]
 be the duality morphism introduced in \cite{Man03} and \cite{Man07}. 
\begin{thm} \label{gluing}Let $X$ be a compact oriented $Spin^c$ 4-manifold with a contact boundary $(Y, \xi)$ and $\s_X$ a $Spin^c$ structure whose restriction on the boundary is compatible with the $Spin^c$ structure induced by $\xi$. Suppose $b_1(X)=0$. 
Then 
\[
\eta \circ ( \Psi(X, \s_X)  \wedge \Psi(Y, \xi)  ) = \Psi(X, \s_{X, \xi}, \xi)
\]
holds. 
\end{thm}
\cref{gluing} implies the following non-triviality of \eqref{ourinv}. 
\begin{cor}
Let $Y$ be a rational homology $3$-sphere equipped with a contact structure $\xi$. If $\xi$ has a symplectic filling with $b_1=0$, then \eqref{ourinv} has a non-equivariant stable homotopy left inverse. In particular, \eqref{ourinv} is not stably null-homotopic. 
Moreover, a left inverse is given by the dual of the relative Bauer-Furuta invariant for the filling. 
\end{cor}

\subsection{KO theoretic obstruction}
When a $4$-manifold is spin, the $S^1$-symmetry of the Seiberg-Witten equation is extended to $Pin(2)$-symmetry, where 
\[
Pin(2):= S^1 \cup j S^1 \subset Sp(1).
\]
 This symmetry has been used in several situations including 10/8-inequality(\cite{F01}), Manolescu's triangulation conjecture(\cite{Man16}) and 10/8-inequality for spin 4-manifolds with boundary(\cite{Ma14}).  In the context of contact topology, F.Lin used $Pin(2)$-symmetry in \cite{L20}.
By the use of \cref{gluing} and $Pin(2)$-equivariant KO-theory, we obstruct a certain class of spin symplectic fillings of contact structures. 

For a contact rational homology $3$-sphere $(Y, \xi)$ with $c_1(\s_\xi)=0$ and a pair 
$
(m, n) \in \Z \times \mathbb{Q}$ with $n + \frac{\sigma(W)}{16} \in \Z$ for a spin bounding $ W$ of $(Y, \s)$, we have two groups 
\[
KOM^{-m,-n}_{Pin(2)}  (-Y, \s_\xi ) :=   \wt{KO}_{Pin(2)} (\Sigma^{  m\wt{\R} \oplus n {\H}   } SWF(-Y, \s_\xi )  )
\]
and its reducible part 
\[
\overline{KOM}^{-m}_{Pin(2)} (-Y, \s_\xi ) :=   \wt{KO}_{Pin(2)} ((\Sigma^{ m \wt{\R} } SWF(-Y, \s_\xi ) )^{S^1} ),  
\]
where $Pin(2)$-action on $\wt{\R}$ and on ${\H}$ are given as the multiplication via $j \mapsto -1$ and restriction of the action of $Sp(1)$.
By the Bott periodicity for the equivariant KO-group, it is sufficient to consider the case that $(m ,n )$ satisfies
\[
(m, n) \in \left\{ (0, l_0), (1, l_1) , (2, l_2 ),  (3, l_3 )  \middle|  l_i \in \left\{0,  \frac{1}{16}, \cdots, \frac{31}{16} \right\} ,~  l_i +  \frac{\sigma (W ) } {16} \in \Z   \right\}. 
\]
We associate a homomorphism 
\[
i^*_{m,n}:KOM^{-m,-n}_{Pin(2)}  (-Y, \s_\xi ) \to \overline{KOM}^{-m}_{Pin(2)} (-Y, \s_\xi )
\]
and 
\[
\varphi_{m} : \overline{KOM}^{-m}_{Pin(2)} (-Y, \s_\xi ) \to \Z
\]
where $i_{m,n}$ is the inclusion map $(\Sigma^{m\wt{\R} } SWF(-Y))^{S^1} \to \Sigma^{m\wt{\R} \oplus n {\H}  } SWF(-Y)$ and the map $\varphi_{m}$ is introduced by Jianfeng Lin in \cite[Definition~5.1]{Lin15}. 
\begin{thm}\label{KO}

We impose either of the following two conditions. 
\begin{itemize}
\item[(i)] 
When 
\[
-d_3(Y, [\xi])-\frac{1}{2}+m+4n\equiv 0,  4\quad\text{ mod }8 
\] 
for $(m, n) \in \Z \times \Q$ with $n +  \frac{\sigma(W)}{16} \in \Z$ for a spin bounding $ W$ of $(Y, \s)$, suppose that the following induced map from $\varphi_{m} \circ i^*_{m,n}$
\[
 (KOM^{-m,-n}_{G}  (-Y, \s_\xi )/\operatorname{Torsion}) \otimes \Z_2 
  \to  \Z_2
\]
is injective. 

\item[(ii)] 
When 
\[
-d_3(Y, [\xi])-\frac{1}{2}+m+4n\equiv 1,  2\quad\text{ mod }8
\] 
for $(m, n) \in \Z \times \Q$ with $n + \frac{\sigma(W)}{16} \in \Z$ for a spin bounding $ W$ of $(Y, \s)$, suppose that the following induced map from $\varphi_{m} \circ i^*_{m,n}$
\[
 KOM^{-m,-n}_{G}  (-Y, \s_\xi ) \otimes \Z_2 
  \to  \Z_2
\]
is injective. 

\end{itemize}

Then any symplectic filling $(X, \om)$ of $(Y, \xi)$ satisfying $ \s_{\om}$ is spin and $b_1(X)=0$ satisfies   
\[
b^+ (X) \leq \mathfrak{e} (m), 
\]
where 
\[
 \mathfrak{e} (m)=
 \begin{cases}
 0&m\equiv 0, 1, 2, 4\text{ mod } 8\\
 1&m\equiv 3, 7\text{ mod } 8\\
 2&m\equiv 6\text{ mod } 8\\
 3&m\equiv 5\text{ mod } 8.
 \end{cases}
\]
In particular, 
\[
b^+ (X) \leq 3. 
\]
\end{thm}
In \cite{OO99}, it is proved that any weak symplectic filling can be modified to a strong symplectic filling. Thus, we do not care about the difference between them.

For example, $-\Sigma(2, 3, 11)$ satisfies the assumption of \cref{KO}. 
 Then for any symplectic filling of a contact structure of $-\Sigma(2, 3, 11)$ such that $\s_\omega$ is spin and $b_1(X)=0$, we have 
 \[
 b^+(X)=1.
 \]
 For the case of Stein fillings of $-\Sigma(2, 3, 11)$, a similar result was proved in \cite{St02}. 
 F.Lin's (\cite{L20}) result is a generalization of the result of $-\Sigma(2,3,11)$ written in \cite{St02}. Note that the result for $-\Sigma(2, 3, 11)$ can be also proved by the argument in \cite[Theorem 3]{L20}.



 

\subsection{Conjecture}
At the end of this section, we write a conjecture related to our invariant. 
\begin{conj}Let 
$\Phi$ be the homomophism 
\[
H_0 (S^0) \to  \widecheck{HM}_{[\xi] }(-Y, \s_\xi) 
\]
obtained as the composition of the following three maps: 
\begin{itemize}
\item[(1)] the map 
\[
\Psi(Y ,\xi )_*  : H_0(S^0) \to  H_0 (\Sigma^{\frac{1}{2}-d_3(-Y, [\xi]) } SWF (-Y, \s_{\xi}) )
\]
induced by $\Psi(Y ,\xi ): S^0 \to \Sigma^{\frac{1}{2}-d_3(-Y, [\xi]) } SWF (-Y, \s_{\xi})$, 
\item[(2)] the map 
\[
H_0 (\Sigma^{\frac{1}{2}-d_3(-Y, [\xi]) } SWF (-Y, \s_{\xi}) )\to  H^{S^1}_0 (\Sigma^{\frac{1}{2}-d_3(-Y, [\xi]) } SWF (-Y, \s_{\xi}) )
\]
induced by 
\[
\Sigma^{\frac{1}{2}-d_3(-Y, [\xi]) } SWF (-Y, \s_{\xi}) ) \wedge ES^1 \to  \Sigma^{\frac{1}{2}-d_3(-Y, [\xi]) } SWF (-Y, \s_{\xi}) ) \wedge_{S^1} ES^1,
\]
 and
\item[(3)] an isomorphism constructed by Lidman-Manolescu(\cite{LM18})
\[
 H^{S^1}_0 (\Sigma^{\frac{1}{2}-d_3(-Y, [\xi]) } SWF (-Y, \s_{\xi}) ) \cong  \widecheck{HM}_{[\xi] }( -Y, \s_\xi) . 
 \]
\end{itemize}
Then 
\[
\Phi (1) = \psi(Y, \xi) \in  \widecheck{HM}_{[\xi]}(-Y, \s_\xi) 
\]
up to sign. 
\end{conj}
\begin{rem}
Although $\psi(Y, \xi)$ is in the $S^1$-equivariant monopole Floer homology $\widecheck{HM}_{[\xi] }(-Y, \s_\xi) $, our invariant is not an $S^1$-equivariant map. This can be understood by the following way: 
We can explicitly give an element $\wt{\psi}(Y, \xi) \in \wt{HM}_{[\xi]}(-Y, \s_\xi)$ such that a natural map 
$\wt{HM}_*(-Y, \s) \to \widecheck{HM}_{* }(-Y, \s)$ sends $\wt{\psi}(Y, \xi)$ to $\psi(Y, \xi)$, where $\wt{HM}_{[\xi]}(-Y, \s_\xi)$ is a flavor of monopole Floer homology introduced in \cite{Bl11}.  In particular, we see that $\psi(Y, \xi)$ is contained in $\ker U \subset \widecheck{HM}_{[\xi] }(-Y, \s_\xi) $ using exact sequence 
\[
\cdots \to \wt{HM}_*(-Y, \s) \to \widecheck{HM}_{* }(-Y, \s) \xrightarrow{U} \widecheck{HM}_{*-2 }(-Y, \s) \to \cdots.
\]
Conjectually, our invariant corresponds to 
\[
\wt{\psi}(Y, \xi) \in  \wt{HM}_{[\xi]} (-Y, \s_\xi) \cong \wt{H}_0(\Sigma^{\frac{1}{2}-d_3(-Y, [\xi]) } SWF(-Y, \s_\xi)) .
\] 
\end{rem}

\subsection{Outline}Here is an outline of the contents of the remainder of this paper: In \cref{Pre},  we first review  Manolescu's Floer homotopy type. In \cref{SWFHC}, we prove a certain boundedness result for the Seiberg-Witten equation in our situation. As a consequence, we define a Seiberg-Witten Floer homotopy contact invariant. We also calculate several Fredholm indices of operators in our situation. 
In \cref{Gluing result}, we prove the gluing theorem of our invariants. We follow the gluing method developed by Manolescu(\cite{Man07}) and Khandhawit-Lin-Sasahira (\cite{KLS18},\cite{KLS'18}). Using the gluing theorem, we give several calculations of our invariants. In \cref{Applications to symplectic fillings}, by the use of the gluing theorem and our invariant, we prove \cref{KO}.

\begin{acknowledgement}
The authors would like to express their deep gratitude to Hirohumi Sasahira for answering their many questions on the paper(\cite{KLS'18}). The authors also wish to thank Mariano Echeverria for answering some questions on his work(\cite{E20}).  The authors also thank Takahiro Oba for discussing examples of symplectic fillings with us. The authors also appreciate Anubhav Mukherjee's helpful comments.  
The first author was supported by JSPS KAKENHI Grant Number 19J23048 and the Program for Leading Graduate Schools, MEXT, Japan.
The second author was supported by JSPS KAKENHI Grant Number 17H06461 and 20K22319 and RIKEN iTHEMS Program.
\end{acknowledgement}

\section{Preliminaries} \label{Pre}

\subsection{Seiberg-Witten Floer homotopy type}
In this subsection, we review Manplescu's construction of the Seiberg-Witten Floer homotopy type. For the detail, see \cite{Man03}. 

Let $Y$ be a rational homology $3$-sphere equipped a $Spin^c$-structure $\s$ and $g$ a Riemann metric on $ Y$.
The spinor bundle with respect to $\s$ is denoted by $S$. When $\s$ is spin, we can regard $S$ as an $Sp(1)$-bundle. 

The map $\rho : \Lambda^*_Y\otimes \C  \to \operatorname{End} (S) $ denotes the Clifford multiplication induced by $\s$. The notation $B_0$ denotes a fixed flat $Spin^c$-connection. 
Then the set of $Spin^c$ connections can be identified with $i\Om^1 (Y)$. 
Then the {\it configuration space} is defined by
\[
\Con_{k-\frac{1}{2}} (Y) := L^2_{k-\frac{1}{2}} (i\Lambda^1_Y) \oplus  L^2_{k-\frac{1}{2}} (S) ,
\]
here $L^2_{k-\frac{1}{2}}$ denotes the completion with respect to $L^2_{k-\frac{1}{2}}$-norm.
In the spin case, we consider the following additional $Pin(2)$-action on $\Con_{k-\frac{1}{2}} (Y)$, where $Pin(2)$ is the subgroup of $Sp(1)$ written as $U(1) \cup j U(1)$: 
\begin{itemize}
\item[(i)] the group $Pin(2)$ acts on $i\Lambda^1_Y$ via the nontrivial homomorphism
$Pin(2) \to O(1)$ and 
\item[(ii)] the group $Pin(2)$ acts on $S$ by the restriction of the natural action of $Sp(1)$.
\end{itemize}

We have the {\it Chern-Simons-Dirac functional }
\begin{align}
CSD: \Con_{k-\frac{1}{2}} (Y) \to \R 
\end{align}
defined by
\[
CSD( b, \psi) := -\frac{1}{2} \int_Y b \wedge db + \frac{1}{2} \int_Y \langle \psi, D_{B_0+ b} \psi \rangle d \operatorname{vol} , 
\]
where $D_{B_0+ b}$ is the $Spin^c$-Dirac operator with respect to the $Spin^c$-connection $B_0+ b$.
The {\it gauge group} 
\[
\G_{k+\frac{1}{2}} (Y) :=\left\{  e^\xi  \middle| \xi \in  L^2_{k+\frac{1}{2}} (Y; i\R )\right\} 
\]
 acts on $\Con_{k-\frac{1}{2}} (Y)$ by
\[
u \cdot (b, \psi) := (b- u^{-1} d u, u \psi). 
\]

 Since the normalized gauge group 
\[
\G^0_{k+ \frac{1}{2} } (Y) :=\left\{  e^\xi  \in\G_{k+\frac{1}{2}} (Y)   \middle|  \int_Y \xi  d\operatorname{vol} =0 \right\}
\] 
freely acts on $\Con_{k-\frac{1}{2}} (Y)$, one can take a slice. The slice is given by 
\[
V_{k-\frac{1}{2}} (Y) :=\ker   \left(d^{*}: L^2_{k-\frac{1}{2}}(i\Lambda^1_Y) \to L^2_{k-\frac{3}{2}} (i\Lambda^0_Y) \right) \oplus L^2_{k-\frac{1}{2}} (S). 
\]
 The formal gradient field of the Chen-Simons Dirac functional with respect to a norm induced by Manolescu is the sum 
 \[
 l+ c:  V_{k-\frac{1}{2}}(Y)\to V_{k-\frac{3}{2} } (Y),
 \]
  where 
  \[
  l(b, \psi) = (*db , D_{B_0} \psi)
  \]
   and 
   \[
   c(b, \psi) = (\pr_{\ker d^{*}}  \rho^{-1} ((\psi \psi^*)_0) , \rho (b) \psi- \xi (\psi) \psi). 
   \]
Here $\xi (\psi ) \in i \Om^0(Y)$ is determined by the conditions 
 \[
 d \xi (\psi) = ( 1- \pr_{\ker d^{*}} ) \circ \rho^{-1} ((\psi \psi^*)_0) \text{ and }  \ \int_Y \xi (\psi) =0. 
 \]
 Note that $l+c$ is $S^1$-equivariant, where the $S^1$-action is coming from 
 \[
 S^1 = \G_{k+\frac{1}{2}} (Y) /  \G^0_{k+\frac{1}{2}} (Y). 
 \]
When $\s$ is spin, we have an additional $Pin(2)$-symmetry. 
 For a subset $I \subset \R$, a map $x= ( b, \psi) : I \to V_{k-\frac{1}{2}} (Y)$ is called a {\it Seiberg-Witten trajectory} if 
 \begin{align}\label{grad}
 \frac{\partial}{\partial t} x(t) = - (l+c) (x(t)). 
 \end{align}
 \begin{defn}
 A Seiberg-Witten trajectory $x = ( b, \psi) : I \to V_{k-\frac{1}{2}} (Y)$ is {\it finite type} if 
 \[
 \sup_{ t\in I} \|\psi(t) \|_{Y}  <\infty \text{ and }\sup_{ t\in I} |CSD (x(t)) |<\infty. 
 \]
 \end{defn}
 We consider subspaces
 $V^\mu_\lambda(Y)$ defined as the direct sums of eigenspaces whose eigenvalues of $l$ are in $(\lambda, \mu]$ for $\lambda< 0 <\mu$ and denote $L^2$-projection from $V_{k-\frac{1}{2}} (Y)$ to $V^\mu_\lambda(Y) $ by $ p^\mu_\lambda$.
  Then the finite dimensional approximation of \eqref{grad} is given by 
  \begin{align}\label{flow}
  \frac{\partial}{\partial t} x(t) = - (l + p^\mu_\lambda c )(x(t)) , 
  \end{align}
  where $x$ is a map from $I \subset \R$  to $V^{\mu}_\lambda (Y)$. 
  Manolescu(\cite{Man03}) proved the following result: 
  \begin{thm}The following results hold. 
  \begin{itemize}
  \item There exists $R>0$ such that all finite type trajectories $x: \R \to V_{k-\frac{1}{2}}(Y)$ are contained in $\overset{\circ}{B}(R;V_{k-\frac{1}{2}}(Y) )$, where $\overset{\circ}{B}(R; V_{k-\frac{1}{2}}(Y))$ is the open ball with radius $R$ in $V_{k-\frac{1}{2}}(Y)$. 
  \item 
  For sufficiently large $ \mu$ and $-\lambda$ and the vector field 
  \[
  \beta (l + p^\mu_\lambda c )
  \]
   on $V^\mu_\lambda(Y)$, $\overset{\circ}{B}(2R;V^\mu_\lambda(Y)) $ is an isolating neighborhood, where $\beta$ is $S^1$-invariant bump function such that $\beta|_{\overset{\circ}{B}(3R)^c}=0$ and  $\beta|_{\overset{\circ}{B}(2R)}=1$. When $\s$ is spin, we take $\beta$ as a $Pin(2)$-invariant function. 
  \end{itemize}
  \end{thm}
  Then an $S^1$-equivariant Conley index $I^\mu_\lambda$ depending on $V^\mu_\lambda(Y) $, the flow \eqref{flow}, an isolating neighborhood $\overset{\circ}{B}(2R)$ and its isolated invariant set is defined. When $\s$ is spin, we take a $Pin(2)$-equivariant Conley index. 
   Then the {\it Seiberg-Witten Floer homotopy type} is defined by 
  \[
 SWF(Y, \s) :=  \Sigma^{-n(Y, \s, g) \C -V^0_\lambda}  I^\mu_\lambda, 
  \]
 as a stable homotopy type of a pointed $S^1$-space,  where $n(Y, \s, g)$ is given by 
  \[
  n(Y, \s, g) := \ind_\C^{APS} (D^+_A ) - \frac{c^2_1( \s_X )-\sigma(X)}{8} . 
  \]
   Here $(X, \s_X)$ is a compact $Spin^c$ bounding of $(Y, \s)$, the used Riemann metric of $X$ is product near the boundary, $\ind_\C^{APS} (D^+_A )$ is the Atiyah-Patodi-Singer index of the operator $D^+_A$ and a $Spin^c$ connection $A$ is a $Spin^c$ connection on X which is an extension of $B_0$. 
  For the meaning of formal desuspensions, see \cite{Man03}. 
  When $\s$ is spin, we set 
\[
 SWF(Y, \s) :=  \Sigma^{-\frac{n(Y, \s, g)}{2} \H -V^0_\lambda}  I^\mu_\lambda, 
  \]
  as a stable homotopy type of a pointed $Pin(2)$-space.

\section{Seiberg-Witten Floer homotopy contact invariant}\label{SWFHC}
 \subsection{Contact structure and conical metric} \label{contact}
 In this subsection, we review geometric settings for constructions of our invariant. 
Let $Y$ be an oriented rational homology $3$-sphere and $\xi$ a positive contact structure on $Y$.
Take a contact $1$-form $\theta$ which is positive on the positively oriented normal field to $\xi$ and a complex structure $J$ on $\xi$ compatible with the orientation. Then we can associate the Riemann metric 
\[
g_1= \theta \otimes \theta + \frac{1}{2} d\theta (\cdot, J \cdot)|_\xi  
\]
 on $Y$. 
On $\R^{\geq 1}  \times Y$, we consider a Riemann metric 
\[
g_0 := ds^2 + s^2 g_1, 
\]
and a symplectic form
\[
\om_0 := \frac{1}{2} d(s^2 \theta), 
\]
where $s$ is the coordinate of $\R^{\geq 1}$. 
This gives an almost K\"ahler structure on $\R^{\geq 1}  \times Y$. 
We consider a metric on 
\[
N^+:= \R^{\geq 0}  \times Y
\]
 which is an extension of $g_0$ and product on $[0,\frac{1}{2}] \times Y$. We also call this metric $g_0$ again.
 The Riemann manifold $N^+$ is what we mainly consider to define our invariant. 
We extend $\om_0$ to a self-dual $2$-form with $|\om_0(s, y)|= \sqrt{2}$ which is translation invariant on $[0, \frac{1}{2}] \times Y$. 
 Then a pair $(g_0, \om_0)$ determines an almost complex structure $J$ on $N^+$. This defines a $Spin^c$ structure 
 \[
 \s:= ( S^+ = \Lambda_{N^+}^{0,0} \oplus  \Lambda_{N^+}^{0,2}, S^-_{N^+} = \Lambda_{N^+}^{0,1} , \rho : T^*N^+ \to \Hom (S^+_{N^+} , S^-_{N^+} ) ),
 \]
  where  
\[
 \rho  = \sqrt{2} \operatorname{Symbol} (\overline{\partial} + \overline{\partial}^* ) . 
 \]
 (See Lemma 2.1 in \cite{KM97}.)
 The notation $\Phi_0$ denotes 
 \[
 (1,0) \in \Om_{\R^{\geq1} \times Y}^{0,0} \oplus  \Om_{\R^{\geq1} \times Y}^{0,2}= \Gamma (S^+|_{\R^{\geq1} \times Y}).
 \]
  We extend $\Phi_0$ to a section of $S^+$ which is zero on $[0, \frac{1}{2}] \times Y$.
 Then the {\it canonical $Spin^c$ connection} $A_0$ on $\s$ is defined by the equation 
 \begin{align}\label{A0}
 D^+_{A_0} \Phi_0= 0
 \end{align}
 on $\R^{\geq 1} \times Y$.
 We also extend $A_0$ to a $Spin^c$ connection which is translation invariant on $[0, \frac{1}{2}] \times Y$. 

 \subsection{Seberg-Witten map}\label{Seberg-Witten map} 
We introduce configuration spaces and gauge groups for 4-manifolds with conical end. We combine Kronheimer-Mrowka's asymptotic condition \cite{KM07} on the conical end of $N^+$ and Khandhawit's double Coulomb slice condition \cite{Khan15} on $\partial N^+$. A technical point is that we use weighted Sobolev spaces to define the double Coulomb slice. First, we define weighted Sobolev spaces.

\subsubsection{Weighted Sobolev norms} 

In this subsection, we give definitions and properties of weighted Sobolev norm on manifolds with conical ends which are also considered in \cite{I19}. 
\begin{defn}  
  A non-compact Riemannian $4$-manifold $(X^+, g_{ X^+})$ (possibly with boundary) is called 
  {\it 4-manifold with a conical end} if $(X^+, g_{ X^+})$  is equipped with a compact subset $K$ in $X^+$ and an isometry between $X^+ \setminus \operatorname{int}K $ and 
  \begin{align}\label{conemetric}
  (\R^{\geq 1}  \times Y, ds^2 + s^2g_Y), 
  \end{align}
   where $Y$ is a closed connected Riemannian $3$-manifold $(Y,g_Y)$. 
  \end{defn}
  We fix an extension $\sigma: X^+ \to \R_{>0}$ of the $s$-coordinate. 
The function $\sigma$ is called {\it a radius function}. 
Let $k$ be a positive integer and $\alpha$ a positive real number.
Let $E$ be a real (complex) vector bundle with an inner product on an oriented $4$-manifold with a conical end $X^+$ and $A$ be a connection on $E$. Then, we use the following family of inner products on $\Gamma_c(E)$: 
\begin{defn}
For any compact support section $s$ of $E$, we define 
\begin{align}\label{inner}
\langle s_1, s_2\rangle_{L^2_{k, \alpha, A} } := \sum_{i=0}^k \int_{X^+} e^{2\alpha \sigma} \langle \nabla^i_A s_1, \nabla^i_A s_2  \rangle  \operatorname{dvol}_{X^+}, 
\end{align}
where the connection $\nabla^i_A$ is the induced connection from $A$ and the Levi-Civita connection. 
 
\end{defn}
The space $L^2_{k, \alpha, A} (E)$ is defined as the completion of $\Gamma_c(E)$ with respect to \eqref{inner}. 
We use the following estimate proved in \cite{I19}. 
\begin{lem}\label{multiplication and compact embedding}
Let $(E_1, |\cdot|_1, A_1), (E_2, |\cdot|_2, A_2)$ be two normed vector bundles equipped with a unitary connection on $X^+$.
Set $W_n=\sigma^{-1}([2^{n-1}, 2^n]) \subset X^+$.
Denote by $\varphi_n: W_1\to W_n$ the diffeomorphism $(t, y)\mapsto (2^{n-1}t,  y)$.
For $i=1, 2$, suppose isomorphisms
\[
(\varphi^*_n E_i)|_{W_1} \cong  E_i|_{W_1}
\]
are given and
there exist constants $a_1, a_2$ such that
\[
|\varphi_n^*s|_{(t, y)}=2^{a_i n}|s|_{\varphi_n(t, y)}
\]
\[
|\nabla^j \varphi^*_n s|_{(t, y)}=2^{(a_i-j)n}|\nabla^j s|_{\varphi_n(t, y)}
\]
for $s \in \Gamma(E_i)$, where we regard $\varphi_n^*s$, $\nabla^j \varphi^*_n $ as  sections of $E_i|_{W_1}$, $(E_i\otimes (T^*X^+)^{\otimes j})|_{W_1}$ respectively by the isomorphism above.
(For example, if $E_1$ is $\Lambda^k$, $|\cdot|_1$ is the norm induced by the Riemannian metric $g_0$, and $A_1$ is the connection induced by the Levi-Civita connection, then an isomorphism $(\varphi^*_n E_1)|_{W_1} \cong  E_1|_{W_1}$ can be given by regarding  $W_n=W_1$ as a manifold (but the metrics are different) and $a_1=-k$ satisfies the condition.)
\begin{enumerate}
\item (Multiplication)\\
For $\alpha_1,\alpha_2 \in \R$, $l\in \Z^{>2}, \varepsilon \in \R^{>0}$, the multiplication 
\[
L^2_{l, A_1, \alpha_1}(E_1)\times L^2_{l, A_2, \alpha_2}(E_2)\to L^2_{l,A_1\otimes A_2, \alpha_1+\alpha_2-\varepsilon}(E_1\otimes E_2)
\]
is continuous.
\item (Compact embedding)\\
For $l \in \Z^{\geq 1},
\alpha'<\alpha$,
the inclusion 
\[
L^2_{l, A_1, 
\alpha}(E_1)\to L^2_{l-1, A_1, 
\alpha'}(E_1)
\] 
is compact.
\end{enumerate}
\end{lem}
The proof of Sobolev multiplication is similar to \cite[Theeorem 13.2.2]{KM07}. 
The proof of Sobolev embedding is the essentially same as the proof of \cite[Theorem 3.12]{Lo87}.

\subsubsection{Seiberg-Witten equation on 4-manifolds with conical end}

Let $Y$ be a rational homology $3$-sphere with a contact structure $\xi$. 
\begin{defn}Let $k$ be a positive integer with $k \geq 4$ and $\alpha$ a positive real number.
We first define the {\it configuration space } $\mathcal{C}_{k, \alpha} (N^+)$ by
\[
\mathcal{C}_{k, \alpha} (N^+) 
:= (A_0, \Phi_0) +  L^2_{k, \alpha}( i\Lambda_{N^+}^1 ) \oplus L^2_{k, \alpha} ( S^+_{N^+}), 
\]
where $L^2_{k, \alpha}( i\Lambda_{N^+}^1 )$ and $L^2_{k, \alpha} ( S^+_{N^+})$ are the completions of the inner products with respect to $L^2_{k, \alpha, \nabla_{LC}}( i\Lambda_{N^+}^1 )  $ and $L^2_{k, \alpha, A_0} ( S^+_{N^+})$.

The gauge group 
$\G_{k+1, \alpha}(N^+)$
is given by 
\begin{align}\label{gauge}
\G_{k+1, \alpha}(N^+) := \left\{ u : N^+ \to \C ~ \middle|~ |u(x)|= 1 ~ \forall x ,~  1-u \in L^2_{k+1, \alpha} (\underline{\C} ) \right\}. 
\end{align}
The action of $\G_{k+1, \alpha}(N^+)$ on $\mathcal{C}_{k, \alpha} (N^+)$  by 
\[
u \cdot (A, \Phi) := (A- u^{-1} du , u \Phi   ). 
\]
We also define the {\it double Coulomb slice} by
\[
\mathcal{U}_{k, \alpha}(N^+)  :=    L^2_{k, \alpha}( i\Lambda_{N^+}^1 )_{CC} \oplus L^2_{k, \alpha} ( S^+_{N^+}) , 
\]
where 
\[
L^2_{k, \alpha}( i\Lambda_{N^+}^1 )_{CC} := \{ a  \in L^2_{k, \alpha}( i\Lambda_{N^+}^1) | d^{*_{\alpha}}  a=0, d^{*} {\bf t}a=0 \} , 
\]
where ${\bf t}$ is the restriction of $1$-forms as differential forms and $d^{*_{\alpha}}$ is the formal adjoint of $d$ with respect to $L^2_{\alpha}$.
\end{defn}
Since $\G_{k+1, \alpha}(N^+) $ can be embedded into $C^0(N^+ , S^1) $, we define the group structure on $\G_{k+1, \alpha}(N^+)$ by multiplication. 

On $N^+$, one can define the {\it Seiberg-Witten map}

\begin{align}
\mathcal{F}_{N^+}  :\mathcal{C}_{k, \alpha} (N^+)  \to  L^2_{k-1, \alpha}( i\Lambda_{N^+}^+ \oplus S^-_{N^+}) 
 \end{align}
by 
\begin{align}\label{SW}
\mathcal{F}_{N^+} (A,  \Phi ) := \left( \frac{1}{2} F^+_{A^t}-\rho^{-1} ( \Phi \Phi^*)_0 - (\frac{1}{2}  F^+_{A^t_0}- \rho^{-1}( \Phi_0 \Phi^*_0)_0 ), D^+_A \Phi  \right)  
\end{align}
where $A_0$ is introduced in \eqref{A0} and $\Phi_0$ is the canonical section and $L^2_{k-1, \alpha}( i\Lambda_{N^+}^+ \oplus S^-_{N^+}) $ is induced by the connection $A_0$. We often omit the Clifford multiplication in our notations. 
When we write $(a, \phi) = (A, \Phi) - (A_0, \Phi_0)$, we can decompose the Seiberg-Witten map $\mathcal{F}_{N^+}$ as the sum of the linear part 
\begin{align}\label{L}
L_{N^+}(a, \phi) := \left( d^+a-(\Phi_0 \phi^*)_0- (\phi \Phi_0^*)_0    , D^+_{A_0} \phi+ \rho(a) \Phi_0   \right),
\end{align}
the quadratic part $C_{N^+}(a, \phi) := (-(\phi \phi^*)_0,  \rho (a)\phi  )$ and 
the constant part $(0, D^+_{A_0} \Phi_0 , 0 )$. 
We sometimes regard $L_{N^+} $ as an operator from $\mathcal{U}_{k, \alpha}(N^+)$  to $L^2_{k-1, \alpha}(i \Lambda_{N^+}^+ ) \oplus L^2_{k-1, \alpha }(S^-_{N^+})$ by the restriction. 
Moreover, the quadratic part is compact by \cref{multiplication and compact embedding}. 
The differential equation 
\begin{align}\label{SW eq}
\mathcal{F}_{N^+} (A,  \Phi )=0 
\end{align}
is called the {\it Seiberg-Witten equation} for $N^+$. The linearlization of $\mathcal{F}_{N^+}$ is given by $
L_{N^+}$.  

In some situations in remaining sections, we also consider 4-manifolds with conical end without boundary. 
We take a compact $Spin^c$ bound $X$ of $Y$. Then we have a glued non-compact manifold 
\[
 X^+ := X \cup_Y N^+
 \]
 without boundary. We use this manifold $X^+$ when we calculate Fredholm indices of elliptic differential operators and prove the gluing theorem.  
Similarly, we have the configuration space 
written by 
\[
\mathcal{C}_{k, \alpha} (X^+) 
:= (A_0, \Phi_0) +  L^2_{k, \alpha}(i \Lambda_{X^+}^1 ) \oplus L^2_{k, \alpha} ( S^+_{X^+}). 
\]
Here a pair $(A_0, \Phi_0)$ on $X^+$ is an extension of $(A_0, \Phi_0)$ on $N^+$. 
We also define the {\it Coulomb slice} by
\[
\mathcal{U}_{k, \alpha}(X^+)  :=    \ker ( d^{*_\alpha} : L^2_{k, \alpha}(i \Lambda_{X^+}^1 ) \to L^2_{k-1, \alpha}(i \Lambda_{X^+}^0 ) ) \oplus L^2_{k, \alpha} ( S^+_{X^+}) . 
\]
 On $X^+$, one can define the {\it Seiberg-Witten map}
\begin{align}
\mathcal{F}_{X^+}  : \mathcal{C}_{k, \alpha} (X^+)  \to L^2_{k-1, \alpha }( i\Lambda^+_{X^+}) \oplus L^2_{k-1, \alpha }( S^-_{X^+})
 \end{align}
by 
\begin{align}\label{SWX}
\mathcal{F}_{X^+} (A,  \Phi ) := \left( \frac{1}{2} F^+_{A^t}- \rho^{-1} ( \Phi \Phi^*)_0 - (\frac{1}{2} F^+_{A^t_0} -\rho^{-1} ( \Phi_0 \Phi^*_0)_0 ), D^+_A \Phi \right) , 
\end{align}
where $d^{*_\alpha}$ denotes the formal adjoint of $d$ with respect to the $L^2_\alpha$-inner product.

\subsection{Hodge decomposition for the double Coulomb subspace}\label{Hodge decomposition for the double Coulomb subspace} 
In this section, we mainly use the Riemann manifold $(N^+,g_0)$ defined in \cref{contact}. 
Note that $N^+$ has a boundary and a conical end. 
We recall the
double Coulomb subspace
\begin{align}
L^2_{k, \alpha} (i\Lambda^1_{N^+} )_{CC} = \{ a \in L^2_{k, \alpha} (i\Lambda^1_{N^+} )  | d^{*_\alpha}a=0, \  d^{*}{\bf t} a=0 \}, 
\end{align}
where ${\bf t}$ is the pull-back as a differential form by the inclusion map $\{0\} \times Y \to N^+$.
We take a compact $Spin^c$ bound $X$ of $Y$. Then we have a glued non-compact manifold 
\[
 X^+ := X \cup_Y N^+. 
 \]
The following proposition is a key lemma to prove the global slice theorem: 
\begin{prop}\label{Slice}
There exists a small positive number $\alpha_0$ depending on a contact form $\theta$ and a complex structure $J$ on $\xi$ such that for any positive real number $\alpha \leq \alpha_0$ satisfying the following conditions: 
\begin{itemize}
\item[(i)] $d^{*_{\alpha}}: L^2_{k, \alpha}(i \Lambda_{X^+}^1 ) \to L^2_{k-1, \alpha}(i \Lambda_{X^+}^0 )$ has closed range, 
\item[(ii)] $d: L^2_{k, \alpha}(i \Lambda_{X^+}^0 ) \to L^2_{k-1, \alpha}(i \Lambda_{X^+}^1 ) $ has closed range, 
\item[(iii)] $\Delta_\al :=d^{*_{\alpha}} \circ  d: L^2_{k, \alpha}(i \Lambda_{X^+}^0 ) \to L^2_{k-2, \alpha}(i \Lambda_{X^+}^0 )  $ is invertible and
\item[(iv)]
we have the following decomposition: 
\begin{align}
L^2_{k, \alpha} (i \Lambda^1_{N^+} ) = L^2_{k, \alpha} (i\Lambda^1_{N^+} )_{CC} \oplus d L^2_{k+1, \alpha} (i\Lambda^0_{N^+} ) . 
\end{align}
\end{itemize}
\end{prop}
\begin{proof}
In order to prove (i), we consider the following operator: 
\begin{align}\label{Lhat'}
\widehat{L}'_{X^+}  :  L^2_{k, \alpha}( i\Lambda_{X^+}^1 \oplus S^+_{X^+}) 
 \to  L^2_{k-1, \alpha}(i\Lambda_{X^+}^0   \oplus i \Lambda_{X^+}^+  \oplus S^-_{X^+}), 
 \end{align}
 given by
  \[
 \widehat{L}'_{X^+} (a, \phi )= ( -d^{*_\alpha}a+ i\operatorname{Re}  \langle i\Phi_0,  \phi \rangle, d^+a-(\Phi_0 \phi^*)_0- (\phi \Phi_0^*)_0  , D^+_{A_0} \phi+ \rho(a) \Phi_0    ) . 
 \]
In \cite[Theorem 3.3]{KM97}, it is proved that $\widehat{L}'_{X^+}$ is Fredholm for $\alpha=0$. Since Fredholm property is an open condition, we can see that there exists a small positive number $\alpha_0$ such that for any positive real number $\alpha \leq \alpha_0$, $\widehat{L}'_{X^+}$ is also Fredholm. The positive number $\alpha_0$ depends only on $\widehat{L}'_{X^+}$ on the end because of the usual parametrix patching argument. Thus, $\alpha_0$ actually depends only on $\theta$ and $J$. Since any Fredholm operator sends a closed subspace to a closed subspace, if we put $\phi=0$, then we can conclude that $d^{*_{\alpha}}: L^2_{k, \alpha}(i \Lambda_{X^+}^1 ) \to L^2_{k-1, \alpha}(i \Lambda_{X^+}^0 )$ has closed range. 
In order to prove (ii), we consider the $L^2_{\alpha}$ formal adjoint 
\[
(\widehat{L}'_{X^+})^*: 
   L^2_{k, \alpha}(i\Lambda_{X^+}^0   \oplus  i \Lambda_{X^+}^+  \oplus S^-_{X^+}) \to L^2_{k-1, \alpha}( i\Lambda_{X^+}^1 \oplus S^+_{X^+}) 
   \]
   of $ \widehat{L}'_{X^+}$ described as  
\[
( \widehat{L}'_{X^+})^* (f, b, \psi ):= (- d f + 2i \operatorname{Im}\psi \otimes \Phi^*_0   ,  d^{*_\al} b , D_{A_0}^- \psi + f \Phi_0 ) , 
\]
where $D_{A_0}^-$ is the $L^2_{\al}$-formal adjoint of $D_{A_0}^+$. (Note that $D_{A_0}^-$ is not the $L^2$-formal adjoint of $D^+_{A_0}$. ) In \cite[Theorem 3.3]{KM97}, it is proved that $( \widehat{L}'_{X^+})^*$ is Fredholm when $\al=0$. Moreover, in \cite{I19}, it is checked that, for $\al \in [0, \al_0]$,  $( \widehat{L}'_{X^+})^*$ is also Fredholm. (If we need, we again take a small number $\al_0$.) This implies, for such a $\al_0$, $\operatorname{Im} d$ is closed. 

Because $\operatorname{Im} d$ is closed, we have the following $L^2_{\alpha}$-orthogonal decomposition: 
\[
L^2_{k, \alpha}(i \Lambda_{X^+}^1 ) = \ker d^{*_{\alpha}} \oplus d (L^2_{k, \alpha}(i \Lambda_{X^+}^0 )) . 
\]
So, $\Delta_\al : L^2_{k, \alpha}(i \Lambda_{X^+}^0 ) \to L^2_{k-2, \alpha}(i \Lambda_{X^+}^0 )  $ has a closed image since $\operatorname{Im}\Delta_\al = \operatorname{Im} d^{*_\alpha} $. Therefore, we also have  the following $L^2_{\alpha}$-orthogonal decomposition: 
\[
L^2_{k, \alpha}(i \Lambda_{X^+}^0 ) = \Delta_\al ( L^2_{k, \alpha}(i \Lambda_{X^+}^0 )) \oplus ( \Delta_\al ( L^2_{k, \alpha}(i \Lambda_{X^+}^0 )) )^{\perp_{L^2_{\alpha}}} 
\]
\[
= \Delta_\al ( L^2_{k, \alpha}(i \Lambda_{X^+}^0 )) \oplus \ker \Delta_\al . 
\]
Moreover, for any element $f \in \ker \Delta_\al$, we can see that
 \[
0=\langle \Delta_\al (f), f \rangle_{L^2_{\alpha}} = \| df \|_{L^2_{\alpha}}.  
\]
So, $f$ is constant and $f \in L^2$, we can conclude that $f=0$. This implies $\Delta_\al : L^2_{k, \alpha}(i \Lambda_{X^+}^0 ) \to L^2_{k-2, \alpha}(i \Lambda_{X^+}^0 )  $ is invertible for $0\leq \alpha \leq \alpha_0$.

Next, we will prove (iv). 

We fist prove 
\begin{align} \label{intersection}
L^2_{k, \alpha} (i\Lambda^1_{N^+} )_{CC}\cap  d L^2_{k+1, \alpha} (i\Lambda^0_{N^+} )  = \{0\}. 
\end{align}
Here we use the connectivity of $\partial N^+$. Let $a=df$ be an element in 
\[
L^2_{k, \alpha} (i\Lambda^1_{N^+} )_{CC}\cap  d L^2_{k+1, \alpha} (i\Lambda^0_{N^+} ).
\]
 Then Green's formula implies 
\begin{align}\label{green}
\langle d 1 , a \rangle_{L^2_\alpha}  - \langle 1 , d^{*_\alpha}a \rangle_{L^2_\alpha} = \int_{\partial N^+} {\bf t} 1 \wedge * {\bf n} a . 
\end{align}
By \eqref{green}, we conclude that 
\begin{align} \label{pre}
 0  = \int_{\partial N^+} * {\bf n} a. 
 \end{align}
 We again use Green's formula and obtain 
 \begin{align} \label{green2}
 \| df\|^2_{L^2_\alpha} - \langle f, d^{*_\alpha}d f \rangle_{L^2_\alpha} = \int_{\partial N^+} {\bf t} f \wedge * {\bf n} df.
 \end{align}
Since
 \[
 0 = d^{*} {\bf t} a =  d^{*} {\bf t} df =  d^{*}d {\bf t} f= \Delta_{\partial N^+} {\bf t} f
 \]
 and $\partial N^+$ is connected, we see that $ {\bf t}f$ is a constant $c$. Moreover, we have $d^{*_\alpha} d f = d^{*_\alpha}a =0$.
Then \eqref{green2} can be computed as 
\[
\| df \|_{L^2_\alpha}^2 =  \int_{\partial N^+} {\bf t} f \wedge * {\bf n} df = c \int_{\partial N^+} * {\bf n} a = 0, 
\]
here we used \eqref{pre}. So we have $a=df=0$. This completes the proof of \eqref{intersection}. 
Next, we will see 
\[
L^2_{k, \alpha} ( i\Lambda^1_{N^+} ) = L^2_{k, \alpha} (i\Lambda^1_{N^+} )_{CC} +  d L^2_{k+1, \alpha} (i\Lambda^0_{N^+} ) . 
\]
 We need to prove that, for any $\alpha \in L^2_{k, \alpha} (i \Lambda^1_{N^+} )$, there exists $\xi \in L^2_{k+1, \alpha} (i\Lambda^0_{N^+} )$ such that $\alpha- d \xi  \in L^2_{k, \alpha} (i\Lambda^1_{N^+} )_{CC}$, i.e. 
 \begin{align*} 
 d^{*_\alpha} d \xi  =  d^{*_\alpha}\alpha \\
 d^{*} {\bf t} d \xi  =  d^{*} {\bf t} \alpha 
 \end{align*}
 hold. These equations are equivalent to 
  \begin{align*} 
 \Delta_\alpha \xi  =  d^{*_\alpha}\alpha \\
 {\bf t}  \xi  =  G_{\partial N^+} d^{*} {\bf t} \alpha , 
 \end{align*}
where $G_{\partial N^+}$ is the Green operator on $\partial N^+$. 
Therefore, we need to prove surjectivity of the map 
\[
\Delta_\alpha(N^+, \partial ) : L^2_{k+1, \alpha} (i\Lambda^0_{N^+} )  \to L^2_{k-1, \alpha} (i\Lambda^0_{N^+} ) \oplus  L^2_{k+ \frac{1}{2}} (i \Lambda^0_{\partial N^+} ),  
\]
defined by 
\[
\Delta_\al(N^+, \partial ) \xi = ( \Delta_\al \xi , {\bf t}\xi ). 
\]
In order to prove this, we first use the excision principle and reduce the surjectivity of $ \Delta_\al(N^+, \partial )$ to calculations of indexes for several Laplacian operators. We follow a method of J.Lin (\cite[Appendix~A]{Lin19}). 

For the excision principle, we consider the double $X^{dbl}:= X \cup_Y (-X)$ of $X$, its Laplacian 
\[
\Delta(X^{dbl}): L^2_{k+1} (i \Lambda^0_{X^{dbl}} )  \to L^2_{k-1} (i \Lambda^0_{X^{dbl}} )
\]
and the Laplacian for $-X$ 
\[
\Delta(-X, \partial ) : L^2_{k+1} (i \Lambda^0_{-X} )  \to L^2_{k-1} (i \Lambda^0_{-X} ) \oplus  L^2_{k+ \frac{1}{2}} (i \Lambda^0_{\partial (-X) } ),  
\]
defined by 
\[
\Delta(-X, \partial ) \xi = ( \Delta \xi , {\bf t}\xi ). 
\]
We also treat the Laplacian for $X^+$ 
\[
\Delta_\al(X^+) := d^{*_\al} \circ d: L^2_{k+1, \al} (i \Lambda^0_{X^+} )  \to L^2_{k-1, \al } (i \Lambda^0_{X^+} ). 
\]
Then for the operators $\Delta_\al(N^+, \partial ) ,\Delta(X^{dbl}),  \Delta(-X, \partial )$ and $\Delta_\al (X^+)$, we have the following excision result: 
\begin{lem}For any $\al \in [0, \al_0]$, we have
\[
\ind \Delta_\al(N^+, \partial ) + \ind \Delta(X^{dbl}) = \ind \Delta_\al(X^+) + \ind \Delta(-X, \partial ). 
\]

\end{lem}
\begin{proof}This is standard excision principle. We omit the proof. For detail, see \cite{Be04} and \cite[Appendix~A]{Lin19}. 
\end{proof}

By (iii), we have 
\[
\ker \Delta_\al(X^+) = \{0\} \text{ and } \Coker \Delta_\al(X^+)=\{0\}. 
\]
Moreover, it is well-known that $\ind \Delta(-X, \partial )= \ind \Delta(X^{dbl})=0$ (\cite{Sch95}). 
This concludes that $\ind \Delta_\al(N^+, \partial ) =0$. Suppose $\Delta_\al(N^+, \partial ) (\xi ) =0 $. 
Green's formula implies that 
 \begin{align} \label{green4}
 \| d\xi\|^2_{L^2_\al} = \langle \xi, d^{*_\alpha}d \xi \rangle_{L^2_\al} =0 .
 \end{align}
So we have $\Ker \Delta_\al(N^+, \partial ) = \{0\}$. This completes the proof of $\Coker  \Delta_\al(N^+, \partial ) =\{0\}$.
This completes the proof of (iv). 
\end{proof}

\subsection{Fredholm theory}\label{Fredholm} 
In this subsection, we will prove the operator \eqref{L} with spectral boundary condition 
\[
L_{N^+} +p^0_{-\infty}\circ r : 
\mathcal{U}_{k, \alpha}(N^+) \to  L^2_{k-1, \alpha}( i\Lambda_{N^+}^+ \oplus S^-_{N^+})\oplus V^0_{-\infty} (\partial N_+) 
\]
is Fredholm for a certain class of weights. 
First we fix a $Spin^c$ bound $(X, \s_X)$ of $Y$ and consider a $Spin^c$ 4-manifold 
\[
X^+ := X \cup_{\partial (N^+)} N^+. 
\]

In order to prove the Fredholm property of $L_{N^+} +p^0_{-\infty}\circ r $, we introduce the following operator on $N^+$:
\begin{align}\label{Lhat}
\widehat{L}_{N^+} +  \widehat{p}^0_{-\infty} \circ  \widehat{r} :  L^2_{k, \alpha}(i \Lambda_{N^+}^1 \oplus S^+_{N^+}) 
 \to  L^2_{k-1, \alpha}(i\Lambda_{N^+}^0   \oplus \Lambda_{N^+}^+  \oplus S^-_{N^+})\oplus \widehat{V}^0_{-\infty}  (\partial N_+), 
 \end{align}
 given by
  \[
 \widehat{L}_{N^+}(a, \phi )= ( d^{*_\alpha}a , d^+a-(\Phi_0 \phi^*)_0- (\phi \Phi_0^*)_0  , D^+_{A_0} \phi+ \rho(a) \Phi_0    ), 
 \]
 where 
 \begin{itemize}
 \item[(i)] the space $\widehat{V}^0_{-\infty}(\partial N_+) $ is the $L^2_{k-\frac{1}{2}}$-completion of the negative eigenspaces of the operator 
 \[
\widehat{l}:= \begin{pmatrix}
0 & -d^{*} & 0  \\
-d & * d & 0 \\ 
0 & 0 & D_{B_0}  \\
\end{pmatrix} : \Om^0_{\partial N^+} \oplus \Om^1_{\partial N^+} \oplus \Gamma (S) \to \Om^0_{\partial N^+} \oplus \Om^1_{\partial N^+} \oplus \Gamma (S), 
\]
 \item[(ii)] the map $\widehat{r} : L^2_{k, \alpha}(i \Lambda_{N^+}^1 \oplus S^+_{N^+})  \to \Om^0_{\partial N^+} \oplus \Om^1_{\partial N^+} \oplus \Gamma (S) $ is the restriction, 
 \item [(iii)]the operator  
 \[
 \widehat{p}^0_{-\infty} : \Om^0_{\partial N^+} \oplus \Om^1_{\partial N^+} \to \widehat{V}^0_{-\infty}  (\partial N_+)
 \]
 is the $L^2$-projection to $\widehat{V}^0_{-\infty}  (\partial N_+)$. 
 \end{itemize}

\begin{lem}\label{fredd}
Suppose $0\leq \alpha \leq \alpha_0$, where $\alpha_0$ is the constant appeared in \cref{Slice}
Then the operator $\widehat{L}_{N^+} +  \widehat{p}^0_{-\infty} \circ  \widehat{r}$ defined in \eqref{Lhat}  is Fredholm for $k \geq 1$. 

\end{lem}
\begin{proof}
In the proof of \cref{Slice}, we confirm that 
\begin{align}\label{Lhat'}
\widehat{L}'_{X^+}  :  L^2_{k, \alpha}(i \Lambda_{X^+}^1 \oplus S^+_{X^+}) 
 \to  L^2_{k-1, \alpha}(i\Lambda_{X^+}^0   \oplus i\Lambda_{X^+}^+  \oplus S^-_{X^+})
 \end{align}
 defined by 
  \[
 \widehat{L}'_{X^+} (a, \phi )= ( -d^{*_\alpha}a+ i\operatorname{Re}  \langle i\Phi_0,  \phi \rangle, d^+a-(\Phi_0 \phi^*)_0- (\phi \Phi_0^*)_0  , D^+_{A_0} \phi+ \rho(a) \Phi_0    ) 
 \]
 is Fredholm for $0 \leq \alpha \leq \al_0$.
Thus we obtain a parametrix of the operator $ \widehat{L}$ on $\R^{\geq 1} \times  Y$.
By the technique \cite{APSI}, one can take the inverse of the AHS operator with the spectral boundary condition on $\R^{\leq 0}\times Y$.  Then the standard patching argument gives a parametrix of $\widehat{L} $ on $N^+$.This proves the Fredholm property of  \eqref{Lhat}. 
This completes the proof. 
\end{proof}
Finally, we prove the Fredholmness result that will be needed to construct our invariant and present  its Fredholm index in terms of the following quantities:
 \[
2n (Y, g_Y, \s) := \ind_\R^{APS} ( D_{A_0}^+) - \frac{c_1^2 (S^+_X) -\sigma(X)}{4} 
 \]
 is  the quantity introduced by Manolescu(\cite{Man03})
 and
 \[
 d_3(Y, [\xi])=\frac{1}{4} (c_1^2 (S^+_X) -2\chi (X) -3\sigma(X))-\langle e (S^+_X, \Psi) , [ X, \partial X] \rangle
 \]
 is the $d_3$ invariant of the homotopy class of the plane field $\xi$, 
where $(X, S^\pm_X, \rho_X)$ is a $Spin^c$ bound of $(Y, \s)$, 
and $\Psi$ is a unit section of $S^+_X|_{Y}$ determined by $\xi$ under the correspondence  of Lemma 2.3 in \cite{KM97}.
Note that when the $Spin^c$ structure of $X$ comes from an almost complex structure $J$  with $\xi=JTY\cap TY$, we can extend $\Psi$ to a nowhere-vanishing section of $S^+_X$ on $X$ and thus $\langle e (S^+_X, \Psi) , [ X, \partial X] \rangle=0$.

\begin{prop}\label{Fredholmind}
For $0\leq \alpha \leq \alpha_0$, where $\alpha_0$ is the constant appeared in \cref{Slice}, $L_{N^+}+p^0_{-\infty}\circ r$ is Fredholm and its index is 
\[
\ind_\R(L_{N^+}\oplus p^0_{-\infty}\circ r)=-d_3(Y, [\xi])-\frac{1}{2}+2n(-Y, g_Y, \s)
\]
\end{prop}
\begin{proof}
This argument is similar to that of \cite{Khan15}, which deals with a compact 4-manifold with boundary instead of $N^+$.
First, by the choice of $\alpha$, \cref{fredd} implies that $\widehat{L}_{N^+}\oplus (\widehat{p}^0_{-\infty}\circ \widehat{r})$ is Fredholm. 
Consider an extra operator
\[
\widehat{L}_{N^+}\oplus ((p^0_{-\infty} \oplus \varpi)\circ \hat{r}): L^2_{k, \alpha}(i\Lambda^1_{N^+}\oplus S^+_{N^+})
\]
\[
\to L^2_{k-1, \alpha}(i\Lambda^0_{N^+} \oplus i\Lambda^+_{N^+}\oplus S^+_{N^+})\oplus V^0_{-\infty}(\partial N^+)\oplus i\R^{b_0(\partial N^+)}\oplus dL^2_{k-1/2}(i\Lambda^0_{\partial N^+})
\]
where
\[
\varpi: \widehat{V}(\partial N^+)\to  i\R^{b_0(\partial N^+)}\oplus dL^2_{k-1/2}(i\Lambda^0_{\partial N^+})
\]
is the  $L^2$-orthogonal projection.

We will show that $L_{N^+}\oplus (p^0_{-\infty}\circ r)$ and $\widehat{L}_{N^+}\oplus ((p^0_{-\infty} \oplus \varpi)\circ \hat{r}$ are Fredholm and 
\[
\ind(L_{N^+}\oplus (p^0_{-\infty}\circ r))=\ind(\widehat{L}_{N^+}\oplus ((p^0_{-\infty} \oplus \varpi)\circ \hat{r}))
\]
\[
=\ind(\widehat{L}_{N^+}\oplus (\widehat{p}^0_{-\infty}\circ \widehat{r}))=-d_3(Y, [\xi])-\frac{1}{2}+2n(-Y, g_Y, \s)
\]
\par
Let 
\[
V^\perp=i\Omega^0(\partial N^+)\oplus id\Omega^0(\partial N^+)
\]
\[
l^\perp: V^\perp\to V^\perp
\]
be the operator 
\[
l^\perp=\begin{bmatrix}
0& -d^{*}\\
-d& 0
\end{bmatrix}.
\]
We denote its $L^2_{k-1/2}$-completion  by the same notation.
Then 
\[
\widehat{V}=V\oplus V^\perp 
\]
and
\[
\widehat{l}=l\oplus l^\perp.
\]
Let
$(V^\perp)^0 _{-\infty}$ be the span of non-positive eigenvectors of $l^\perp$.

As shown in \cite{Khan15}, the map
\[
\varpi: (V^\perp)^0 _{-\infty} \to i\R^{b_0(\partial N^+)}\oplus d L^2_{k-1/2}(i\Lambda^0_{\partial N^+})=: W( \partial N^+ )
\]
is an isomorphism. 
Thus,  the commutative diagram
\[
  \begin{CD}
     L^2_{k, \alpha}(i\Lambda^1_{N^+}\oplus S^+_{N^+})@>{\widehat{L}_{N^+}\oplus \widehat{p}^{0}_{-\infty} \circ \widehat{r}}>> L^2_{k-1, \alpha}(i\Lambda^0_{N^+}\oplus i\Lambda^+_{N^+}\oplus S^-_{N^+})\oplus \widehat{V}^0_{-\infty} \\
@\vert  @V{id\oplus \varpi}V{\cong}V    \\
   L^2_{k, \alpha}(i\Lambda^1_{N^+}\oplus S^+_{N^+})  @>{\widehat{L}_{N^+}\oplus ((p^0_{-\infty} \oplus \varpi)\circ \hat{r})}>>  L^2_{k-1, \alpha}(i\Lambda^0_{N^+}\oplus i\Lambda^+_{N^+}\oplus S^-_{N^+})\oplus V^0_{-\infty}\oplus  W( \partial N^+ )
  \end{CD}
\]
implies $ \widehat{L}_{N^+}\oplus ((p^0_{-\infty} \oplus \varpi)\circ \hat{r}$ is Fredholm and
\[
\ind(\widehat{L}_{N^+}\oplus ((p^0_{-\infty} \oplus \varpi)\circ \hat{r}))=\ind(\widehat{L}\oplus (\widehat{p}^0_{-\infty}\circ \widehat{r})). 
\]

First, we show
\[
\ind(L_{N^+}\oplus (p^0_{-\infty}\circ r))=\ind(\widehat{L}_{N^+}\oplus ((p^0_{-\infty} \oplus \varpi)\circ \hat{r})).
\]
We put 
\[
W( \partial N^+ ) := i\R^{b_0(\partial N^+)}\oplus d L^2_{k-1/2}(i\Lambda^0_{\partial N^+}) .
\]

We can apply the snake lemma to the following diagram:
\[
\begin{CD}
0  @. 0 \\ 
  @V VV  @VVV    \\
L^2_{k, \alpha}(i\Lambda^1_{N^+}\oplus S^+_{N^+})_{CC}@>{L_{N^+}\oplus p^0_{-\infty}\circ r}>>  L^2_{k-1, \alpha}(i\Lambda^+\oplus S^-_{N^+})\oplus V^0_{-\infty} \\ 
 @V{} VV  @V{}VV    \\
L^2_{k, \alpha}(i\Lambda^1_{N^+}\oplus S^+_{N^+}) @>{\widehat{L}_{N^+}\oplus ((p^0_{-\infty} \oplus \varpi)\circ \hat{r}) }>>    L^2_{k-1, \alpha}(i\Lambda^0_{N^+}\oplus i\Lambda^+_{N^+}\oplus S^-_{N^+})\oplus V^0_{-\infty}\oplus W( \partial N^+ ) \\ 
   @V{d^{*_\alpha}\oplus \varpi \circ \widehat{r}} VV  @VVV    \\
 L^2_{k, \alpha}(i\Lambda^0) \oplus W( \partial N^+ )@=   L^2_{k-1, \alpha}(i\Lambda^0_{N^+})\oplus W( \partial N^+ )   \\ 
 @V VV  @VVV    \\
 0  @. 0 \\ 
 \end{CD}
\]
 Thus, we obtain
 \[
 \ind(L_{N^+}\oplus (p^0_{-\infty}\circ r))=\ind(\widehat{L}_{N^+}\oplus ((p^0_{-\infty} \oplus \varpi)\circ \hat{r})).
 \]
 \par
Next, we show 
\[
\ind(\widehat{L}_{N^+}\oplus ((p^0_{-\infty} \oplus \varpi)\circ \hat{r}))=\ind(\widehat{L}_{N^+}\oplus (\widehat{p}^0_{-\infty}\circ \widehat{r})).
\]

Finally, we show
\[
\ind(\widehat{L}_{N^+}\oplus (\widehat{p}^0_{-\infty}\circ \widehat{r}))=
-d_3(Y, [\xi])-\frac{1}{2}+2n(-Y, g_Y, \s).
\]

Let $X$ be a $Spin^c$ bound of $(Y, \xi)$ and $X'$ be a $Spin^c$ bound of $(-Y, \s)$.
Then, by the excision property of index, 
\[
\ind_\R (\widehat{L}_{N^+}\oplus \widehat{p}^{0}_{-\infty} \circ \widehat{r})+\ind_\R (\widehat{L}_{X\cup_Y X'})\\
=\ind_\R(\widehat{L}_{X^+})+
\ind_\R (\widehat{L}_{X'}\oplus \widehat{p}^{0}_{-\infty} \circ \widehat{r})
\]
holds.
Thus, we have
\begin{align*}
&\ind_\R (\widehat{L}_{N^+}\oplus \widehat{p}^{0}_{-\infty} \circ \widehat{r})-
\ind_\R (\widehat{L}_{X'}\oplus \widehat{p}^{0}_{-\infty} \circ \widehat{r})\\
=&\ind_\R(\widehat{L}_{X^+})-\ind_\R (\widehat{L}_{X\cup_Y X'})\\
=&\langle e(S^+_X, \Psi_\xi), [X, \partial X]\rangle- \frac{c^2_1(S^+_{X})-2\chi(X)-3\sigma(X)}{4}-\frac{c^2_1(S^+_{X'})-2\chi(X')-3\sigma(X')}{4}\\
=&-d_3(Y, [\xi])-\frac{c^2_1(S^+_{X'})-2\chi(X')-3\sigma(X')}{4}, 
\end{align*}
here we use the computation result of the index of $\wh{L}_{X^+}$ given in \cite[Theorem 2.4]{KM97}. 
Therefore,
\[
\ind_\R (\widehat{L}_{N^+}\oplus \widehat{p}^{0}_{-\infty} \circ \widehat{r})=
\]
\[
-d_3(Y, [\xi])+\ind^{APS}_{\R}(d^{*_\alpha}+d^++D^+_{A_0})_{X'}-\frac{c^2_1(S^+_{X'})-2\chi(X')-3\sigma(X')}{4}
\]
Now, from the index formula (for example, see \cite[Section 3]{Khan15})
\[
\ind^{APS}_{\R}(d^{*_\alpha}+d^+)_{X'}=-\frac{\sigma(X')+\chi(X')}{2}-\frac{1}{2}
\]
and the definition  
\[
2n(\partial X', g, \s):=\ind^{APS}_{\R}(D^+_{A_0})_{X'}-\frac{c^2_1(S^+_{X'})-\sigma(X')}{4}, 
\]
we have
\begin{align*}
&\ind^{APS}_{\R}(d^{*_\alpha}+d^++D^+_{A_0})_{X'}-\frac{c^2_1(S^+_{X'})-2\chi(X')-3\sigma(X')}{4}\\
=&\left\{\ind^{APS}_{\R}(d^{*_\alpha}+d^+)_{X'}+\frac{  \sigma(X')+\chi(X') }{2}\right\}+\left\{\ind^{APS}_{\R}(D^+_{A_0})_{X'}-\frac{c^2_1(S^+_{X'})-\sigma(X')}{4}\right\}\\
=&-\frac{1}{2}+2n(\partial X', g, \s).
\end{align*}
Thus, we obtain 
\[
\ind_\R (\widehat{L}_{N^+}\oplus \widehat{p}^{0}_{-\infty} \circ \widehat{r})=-d_3(Y, [\xi])-\frac{1}{2}+2n(-Y, g_Y, \s).
\]

\end{proof}

\subsection{Uniform bound for energies}
In this subsection, we prove a certain boundedness of the solutions of the Seiberg-Witten equation on $N^+$. This is a main step to construct our Floer homotopy contact invariant.

We consider a half-cylinder $\R^{\leq 0}\times Y$ with the product metric and a $Spin^c$ structure, $A_0$ and $\Phi_0$ on it as translation invariant, which are the same as $\s|_{[0, \frac{1}{2} ] \times Y}$, $A_0|_{{[0, \frac{1}{2} ] \times Y}}$ and $\Phi_0|_{{Y \times [0, \frac{1}{2} ] }}=  0$.  
The notations $S^+_{ \R^{\leq 0}\times Y}$, $S^-_{\R^{\leq 0}\times Y}$ denote the spinor bundles. 

Our main result in this section is: 
\begin{thm}\label{bounded}
There exists $0<\al_1\leq \alpha_0$ depending only on $\theta$ and $J$ such that the following conclusion holds, where $\alpha_0$ is the constant appeared in \cref{Slice}. 
Let $\al$ be an element in $(0,\alpha_1]$. 
There exists a constant $R'>0$ independent of $\al$ such that the following result holds. 
Suppose that 
 \[
(x, y ) \in  \cU_{k, \alpha}  ({N^+}) \times L^2_{k} (i\Lambda^1_{\R^{\leq 0}\times Y }\oplus S^+_{\R^{\leq 0}\times Y}  )
\]
satisfies the following conditions: 
\begin{itemize}
\item[(i)] the element $x+ (A_0, \Phi_0)$ is a solution of \eqref{SW eq} on $N^+$, 
\item[(ii)]the element $y$ is a solution of Seiberg-Witten equation on $\R^{\leq 0}\times Y$,
\item[(iii)] $y$ is temporal gauge, $d^{*}b(t)=0$ for each $t$, where $y(t)= (b(t), \psi(t))$ and $y$ is finite type and 
\item[(iv)] $x|_{\partial N^+} = y (0)$. 
\end{itemize}
  Then 
  \[
  \| x\|_{L^2_{k, \alpha} } < R'  \text{ and } \|y(t)\|_{L^2_{k-\frac{1}{2}}} < R' \  (\forall t  \leq 0). 
  \]
\end{thm}

In order to prove \cref{bounded}, we use several corresponding notions used in \cite{Man03} and \cite{KM97}. 
\begin{defn}We consider a Riemannian manifold 
\[
N^+_* := \R^{\leq 0} \times Y \cup N^+
\]
obtained by gluing the half-cylinder $(\R^{\leq 0} \times Y, dt^2+ g_Y)$ and $N^+$ along their boundary. 
 The solutions $(A, \Phi)$ of the Seiberg-Witten equation on $N^+_*$ are called $N^+_*$-trajectories. If an $N^+_*$ trajectory $(A, \Phi)$ satisfies 
\[
\sup_{t \in \R^{\leq 0}} |CSD (A|_{\{t\} \times Y } ) |< \infty \text{ and } \|\Phi\|_{C^0( \R^{\leq 0} \times Y )} < \infty, 
\]
then $(A, \Phi)$ is called a {\it finite type $N^+_*$-trajectory}. 
\end{defn}
We also use a notion of the notion of energy introduced in \cite{KM97}: 
For an element $(A, \Phi) \in \mathcal{C}_{k, \al} (N^+_*)$, we regard $(A, \Phi)|_{\R^{\geq 1} \times Y}$ as an element 
\[
(a, \alpha , \beta ) \in i\Omega^1_{\R^{\geq 1} \times Y} \oplus \Om^{0,0}_{\R^{\geq 1} \times Y} \oplus  \Om^{0,2}_{\R^{\geq 1} \times Y}.
\]
 For this description and a suitable subset $U$ in $\R^{\geq 1} \times Y$, we define 
\begin{align}
E_{U} (A, \Phi)  := \int_U (1- |\alpha|^2-   |\beta|^2)^2 + |\beta|^2 + |\nabla_a \alpha| ^2 + | \widetilde{\nabla}_a\beta |^2 , 
\end{align}
where $\widetilde{\nabla}_a$ is the unique unitary connection in $\Lambda^{0,2}$ whose $(1,0)$-part is equal to $\partial_a$ under the identification $\Lambda^{1,0} \otimes \Lambda^{0,2} =\Lambda^{1,2}$. 
In order to prove \cref{bounded}, we need the following four propositions. We first see the exponential decay estimate for the energy. 
\begin{prop}[\cite{KM97}, Proposition 3.15 ]\label{1}For any constant $E_0>0$, there exists constant $\varepsilon_{E_0}>0$ and $C_{E_0}>0$, such that if $(A, \Phi)$ is a solution of $\eqref{SW eq}$ satisfying $E_{\R^{\geq 1} \times Y} (A, \Phi) \leq E_0$, then 
\[
E_{ [s, s+1]  \times Y } (A, \Phi) \leq C_{E_0}e^{-\varepsilon_{E_0} s}, 
\]
for any $s \geq 1$.
\end{prop}
Note that our constants $C_{E_0}$ and $\varepsilon_{E_0}$  depend on $E_0$. Later, we will prove that $C_{E_0}$ and $\varepsilon_{E_0}$ do not depend on $E_0$
The proof is the completely same as Proposition 3.15 in \cite{KM97}. Next, we see a bound of spinors for finite type $N^+_*$-trajectories. 
\begin{prop}[\cite{KM97}, Lemma 3.14]\label{2}
There exists a constant $\kappa$ such that for all finite type $N^+_*$-trajectories $(A, \Phi)$, we have 
\[
\sup_{x \in N^+_*} | \Phi(x) | ^2 \leq \kappa . 
\]
\end{prop}
\begin{proof}
Put $S:= \sup_{x \in N^+_*} | \Phi(x) |$.  We consider the following two cases: 
\begin{itemize}
\item[(1)] $S=  \max_{x \in N^+_*} | \Phi(x) |$
\item[(2)] There is no points satisfying $S= | \Phi(x) |$.
\end{itemize}
In the first case, by the standard argument in the case of closed 4-manifolds, we have 
\[
S < \| \operatorname{Scal} (N^+_*) \|_{C^0} . 
\]
Note that $\| \operatorname{Scal} (N^+_*) \|_{C^0}$ is bounded since we are considering product and cone metrics. 
In the second case, one can take a sequence of points $\{x_n\} \subset N^+_*$ such that $|\Phi(x_n)|  \to S$.
By taking a subsequence, we can reduce to the following two cases: 
\begin{itemize}
\item[(2)-(i)] $x_i \in \R^{\leq -i} \times Y$ and
\item[(2)-(ii)] $x_i \in \R^{\geq i } \times Y$. 
\end{itemize}
In the second case, since we have $\|\Phi-\Phi_0\|_{L^2_{3, \alpha} (\R^{\geq 1 } \times Y) } < \infty$, so $\|\Phi-\Phi_0\|_{C^0(\{s\} \times Y) } \to 0$ as $s\to \infty$ by Sobolev embedding theorem on $[i, i+1] \times Y$. Here we use $\al>0$.
Therefore, in this case, we have 
\[
S \leq 1. 
\]
In the first case, by the same discussion in the proof of \cite[Proposition 1]{Man03}, we have 
\[
S \leq \| \operatorname{Scal} (Y, g|_Y ) \|_{C^0} 
\]
This implies the conclusion. 
\end{proof}
As the third proposition, we consider a universal bound of  the energies of finite type trajectories. 
\begin{prop}[\cite{KM97}, Lemma 3.17, \cite{MR06}, Lemma 2.2.7]\label{3}
There exist a constant $\kappa$ and a positive integer $i_0$ such that 
for any finite type $N^+_*$-trajectory $(A, \Phi)$ of the Seiberg-Witten equation on $N^+_*$, we have 
\[
E_{\R^{\geq i_0 } \times Y } (A, \Phi ) \leq \kappa .
\]
\end{prop}
\begin{proof}
First, we follow the proof of Lemma 3.17. 
We fix a positive integer $i_0$ such that 
\[
|N_J|_{C^0 ( [i_0, i_0+1] \times Y )} \leq \frac{1}{32}
\]
 and 
 \[
 |F^\om_{\widetilde{\nabla} } |_{C^0 ( [i_0, i_0+1] \times Y )} \leq \frac{1}{8},
 \] 
where $N_J$ is the Nijenhuis tensor of $J$, $\widetilde{\nabla}$ is the unique unitary connection in $\Lambda^{0,2}$ whose $(1,0)$-part is equal to $\partial$ and $F^\om_{\widetilde{\nabla} }  =\frac{1}{2}\langle F_{\widetilde{\nabla} }  , \om\rangle$. Then, in \cite{KM97}, it is proved that 
\[
E_{\R^{\geq i_0+2 } \times Y } \leq {\kappa}' + \int_{ \partial (\R^{\geq i_0+2 } \times Y )  } \frac{1}{4} i a|_{\{i_0+2\} \times Y } \wedge \om|_{\{i_0+2\} \times Y }
\]
for some constant $\kappa'$.
Next, we consider a cut off of the connection $a$. Let $a'$ be a connection given by $\rho a$, where $\rho$ is a cut off function satisfying $\rho|_{ \R^{\leq 0} \times Y } =0$ and $\rho|_{ \R^{\geq 1}  \times Y } =1$. We also extend the closed form $\om$ by $\om:= \frac{1}{2} d (\rho  s^2 \theta ) $.
Then the integration 
\[
\int_{ \partial (\R^{\geq i_0+2 } \times Y )  } \frac{1}{4} i a|_{\{i_0+2\} \times Y } \wedge \om|_{\{i_0+2\} \times Y }
\]
 can be regarded as 
\[
-\int_{ [0,  i_0+2 ] \times Y   } \frac{1}{4} i da' \wedge \om 
\]
by the Stokes theorem. By the Peter-Paul inequality, we have 
\begin{align*}
\frac{1}{4} \left| \int_{ [0,  i_0+2 ] \times Y   }  da' \wedge \om \right| &\leq \frac{1}{8}\left( \int_{[0,  i_0+2 ] \times Y } |da'|^2 + \int_{[0,  i_0+2 ] \times Y } |\om |^2  \right) \\
&\leq \frac{1}{8}\left( \int_{N^+_*} |da'|^2 + \int_{[0,  i_0+2 ] \times Y } |\om |^2  \right) \\
&\leq \frac{1}{8}\left( 2 \int_{N^+_*} |d^+ a'|^2 + \int_{[0,  i_0+2 ] \times Y } |\om |^2  \right).  \\
\end{align*}
Then the Seiberg-Witten equation and \cref{2} imply 
\[
\frac{1}{8}\left( 2 \int_{N^+_*} |d^+ a'|^2 \right) < c' .
\]
Thus, we conclude that 
\[
E_{\R^{\geq i_0+2 } \times Y } \leq {\kappa}' + c' + \int_{[0,  i_0+2 ] \times Y } |\om |^2. 
\]
This completes the proof. 
\end{proof}
As the fourth proposition, we estimate analytic energies $\mathcal{E}^{an}$ for finite type trajectories. For the definition of $\mathcal{E}^{an}$, see \cite[Definition 4.5.4]{KM07}. 
 \begin{prop}\label{4}  Let $i_0$ be the positive integer appeared in \cref{3}. 
 There exists a positive integer $i_0$ such that the following result holds.  
 There exists a constant ${\kappa}'$ such that any finite type $N^+_*$-trajectory $(A, \Phi)$, we have 
\[
\mathcal{E}^{an}_{\R^{\leq i_0 +1}\times Y} (A, \Phi) \leq {\kappa}'. 
\]
 \end{prop}
 \begin{proof}
As in the proof of Proposition 1 in \cite{Man03}, we can suppose that there exists a solution of the 3-dimensional Seiberg-Witten equation  $(B_{-\infty}, \Psi_{-\infty})  $ such that 
\[
\| (A, \Phi)  - \pr^* (B_{-\infty}, \Psi_{-\infty})\|_{C^k ( [-i-1, -i ] \times Y ) } \to 0 \text{ as } i \to \infty , 
\]
where $\pr : \R \times Y \to Y$ is the projection.
This implies 
\[
\mathcal{E}^{an}_{\R^{\leq i_0 +1}\times Y} (A, \Phi)
\]
\[
 = 2CSD_Y (B_{-\infty}, \Psi_{-\infty}) -2 CSD_{Y}  (A|_{ \{i_0 +1\} \times Y }, \Phi|_{ \{i_0 +1\} \times Y } ). 
\]
Note that the set of critical values of CSD is uniformly bounded. (Here we use the condition that $Y$ is a rational homology $3$-sphere.)
Moreover, Kronheimer-Mrowka proved that there exist $ \kappa_k$ and $\varepsilon>0$ which is independent of $(A, \Phi)$ and  an $L^2_{k+1}$-gauge transformation $g_{(A, \Phi)}$ such that 
\begin{align}\label{estimate of A}
 \|g_{(A, \Phi)}^* (A, \Phi) - (A_0, \Phi_0) \|_{L^2_k ({ [s, s+1]  \times Y })} \leq \kappa_k e^{-\varepsilon s}. 
 \end{align}
 Actually, Kronheimer-Mrowka proved this result for $L^2_k$-solutions. Since our Sobolev space $L^2_{k, \al}$ is contained in $L^2_k$, we obtain the same result. Moreover, by the condition \eqref{estimate of A}, we can see that $g_{(A, \Phi)}$ is in $L^2_{k+1, \al}$.  
 In our case, since $Y$ is a rational homology 3-sphere, CSD is gauge invariant so 
 \[
 2 CSD_{Y}  (A|_{ \{i_0 +1\} \times Y }, \Phi|_{ \{i_0 +1\} \times Y } )
 \]
  is bounded.
 This completes the proof. 
 \end{proof}
We now give a proof \cref{bounded} by assuming \cref{3} and \cref{4}. 

\begin{proof}[Proof of \cref{bounded}]
Suppose that  \[
(x, y ) \in  \mathcal{U}_{k, \alpha} \oplus ( L^2_{k} (i\Lambda^1(\R^{\leq 0} \times Y )  \oplus L^2_k (S^+_{\R^{\leq 0}\times Y}  ) )
\]
satisfies the assumption of \cref{bounded}. 
We have three steps in this proof. 
{Step 1} 
First, we see that $(x, y )$ defines a finite type $N^+_*$-trajectory $(A , \Phi)$. 
\begin{proof}[Proof of Step 1]This is essentially the same as the proof of Corollary 2 in \cite{Khan15}.
\end{proof} 

{Step 2} 
Second, there exists $0<\alpha_1 \leq \al_0$ such that we obtain a gauge transformation $u$ on $N^+_*$  such that 
\[
\sup_{t \in \Z_{<0}}  \|u^* (A , \Phi)- (A_0, \Phi_0) \|_{L^2_{k } ([t, t+1] \times Y)   } + \|u^* (A , \Phi)- (A_0, \Phi_0) \|_{L^2_{k, \alpha_1} (N^+)   } \leq C, 
\]
where $C$ is independent of $(A , \Phi)$. Here $\al_0$ is the constant given in \cref{Slice}. 
\begin{proof}[Proof of Step 2]
Kronheimer-Mrowka (Theorem 5.1.1 in \cite{KM97}) proved that there exists a gauge transformation $u^+$ on $\R^{\geq i_0} \times Y$ such that for $s \geq i_0$, 
\[
\|(u^+)^*  (A , \Phi)- (A_0, \Phi_0) \|_{L^2_{k } ([s, s+1] \times Y )   } \leq c e^{-\al_1 s} 
\]
for some $\al_1>0$ and $c>0$. By \cite[Lemma 3.21, Lemma 3.22]{KM97} and \cref{3}, we know that $\al_1$ and $c$ depend only on $\theta$ and $J$.
 Therefore, we have 
\[
\|(u^+)^*  (A , \Phi)- (A_0, \Phi_0) \|_{L^2_{k, \alpha_1} (N^+)   } \leq  \kappa'
\]
for some $\kappa'$. 
On the other hand, we can take $u^-$ on $\R^{\leq i_0+1} \times Y$ such that, for $t+1 \leq i_0+1$,  
\[
\sup_{t \in \Z_{<0}}  \|(u^-)^* (A , \Phi)- (A_0, \Phi_0) \|_{L^2_{k } ([t, t+1] \times Y)   } \leq c' 
\]
for some constant $c'$ by using \cref{4}. Here we use the Coulomb slice and the standard bootstrapping argument.
Then, by the standard patching argument for $u^+$ and $u^-$, we obtain a gauge transformation $u $ on $N^+_*$ satisfying our conclusion. 
\end{proof}

{Step 3} 
Suppose $0<\al \leq \al_1$. 
In the third step, we show that the action 
\[
  \mathcal{U}_{k, \alpha}\times \G_{k+1, \alpha}(N^+)  \to \Con_{k, \alpha}(N^+) 
\]
given by 
\[
 u \cdot ( a,\phi  ) := (a- u^{-1} du + A_0, u \phi + u\Phi_0)
\]
is a $\G_{k+1, \alpha}(N^+) $-equivariant diffeomorphism. 
We also apply the Coulomb projection and obtain 
\[
\|x\|_{L^2_{k, \alpha}} < R' 
\]
and 
\[
\| y(t) \|_{L^2_{k-\frac{1}{2}}} < R'   \ \  \forall t \leq 0. 
\]
\begin{proof}[Proof of Step 3]
The first statement is a consequence of \eqref{Slice}. 
The second inequality 
\[
\| y(t) \|_{L^2_{k-\frac{1}{2}}} < R'   \ \  \forall t \leq 0
\]
is followed by Step 2. 
By Step 2, we have the bound
\[
\sup_{t \in \Z_{<0}}  \|u^* (A , \Phi)- (A_0, \Phi_0) \|_{L^2_{k } ([t, t+1] \times Y)   } + \|u^* (A , \Phi)- (A_0, \Phi_0) \|_{L^2_{k, \alpha_1} (N^+)   } \leq C
\]
for some gauge transformation $u$. 
Let $\al$ be a positive real number with $\al \leq \alpha_1$.
Then we consider the projection using the decomposition \cref{Slice}
\[
P : \mathcal{C}_{k, \alpha} (N^+) \to \mathcal{U}_{k, \alpha} (N^+). 
\]
Note that $P$ is not $L^2$, $L^2_{\alpha}$, or $L^2_{k, \alpha}$-orthogonal projection. 
Since $P$ is continuous, we see that 
\[
\| P u^*( A, \Phi)|_{N^+}   \|_{L^2_{k, \alpha}} \leq c\| u^*(A, \Phi)|_{N^+} -(A_0, \Phi_0) \|_{L^2_{k, \alpha}}
\]
for some constant $c$.
This inequality implies
\[
\| P u^*( A, \Phi)|_{N^+}   \|_{L^2_{k, \alpha}} \leq c c_1 
\]
for a constant $c_1$. Since $\mathcal{U}_{k, \alpha}\times \G_{k+1, \alpha}(N^+)  \to \Con_{k, \alpha}(N^+)$ is a diffeomorphism, we have 
\[
P u^* ( A, \Phi) = (A, \Phi) \text{ on } N^+. 
\]
This completes the proof.

\end{proof}

This completes the proof of \cref{bounded}.

\end{proof}

\subsection{Seiberg-Witten Floer homotopy contact invariant}\label{MMMM}
In this section, by the use of boundedness result \cref{bounded}, we construct a Seiberg-Witten Floer homotopy contact invariant. 
To carry out this, we consider a finite-dimensional approximation of the map 
\[
\mathcal{F}_{N^+}  : \mathcal{U}_{k, \alpha}
 \to  \mathcal{V}_{k-1, \alpha} \oplus V (\partial N_+) , 
 \]
 where $\mathcal{U}_{k, \alpha}= L^2_{k, \alpha}( i\Lambda_{N^+}^1 )_{CC} \oplus   L^2_{k, \alpha} (S^+_{N^+})$ and $\mathcal{V}_{k-1, \alpha} = L^2_{k-1, \alpha}( i\Lambda_{N^+}^0\oplus i\Lambda_{N^+}^+ ) \oplus L^2_{k-1, \alpha}(S^-_{N^+})$.
In this section, we fix a weight $\alpha \in (0, \infty)$ satisfying $\alpha \leq \al_1$, where $\al_1$ is the constant appeared in \cref{bounded}. 
  Take sequences of subspaces 
 \[
  \mathcal{V}_1 \subset  \mathcal{V}_2 \subset \cdots \subset  \mathcal{V}_{k-1, \alpha} \text{ and } V^{\lambda_1}_{-\lambda_1} \subset V^{\lambda_2}_{-\lambda_2} \subset \cdots   \subset V (\partial N_+)  
  \]
  such that 
  \begin{itemize}
  \item[(i)] $(\im L_{N^+} + p^{\lambda_n}_{-\lambda_n}\circ r)^{\mathcal{V}_{k-1, \alpha} \oplus V (\partial N_+) } \subset \mathcal{V}_n \oplus  V^{\lambda_n}_{-\lambda_n} (\partial N_+)  $ 
 \item[(ii)] the $L^2$-projection $P_n :  \mathcal{V}_{k-1, \al}  \oplus V (\partial N_+) \to \mathcal{V}_n \oplus  V^{\lambda_n}_{-\lambda_n} (\partial N_+)$ satisfies 
 \[
 \lim_{n \to \infty} P_n (v) =v 
 \]
 for any $ v \in \mathcal{V}_{k-1, \al}  \oplus V (\partial N_+)$.
  \end{itemize}
  Then we define a sequence of subspaces 
  \[
  \mathcal{U}_n :=  (L_{N^+}+ p^{\lambda_n}_{-\lambda_n}\circ r)^{-1} ( \mathcal{V}_n \oplus  V^{\lambda_n}_{-\lambda_n} ). 
  \]
 This gives a family of the approximated Seiberg-Witten map is given by 
  \[
  \{  \mathcal{F}_n : =  P_n (L_{N^+}+C_{N^+}, p^{\lambda_n}_{-\lambda_n} \circ r ) \colon   \mathcal{U}_n \to \mathcal{V}_n \oplus  V^{\lambda_n}_{-\lambda_n} (\partial N_+) \} . 
  \]
  In order to define a cohomotopy type invariant, we need to prove the following proposition: 
 \begin{prop}\label{cohomotopy1}For a large $n$ and a large positive real number $ R$, 
 there exists an index pair $(N_n,  L_n)$ of $V^{\lambda_n}_{-\lambda_n} (\partial N_+) $ and a sequence $\{\varepsilon_n\}$ of positive numbers such that 
 \begin{align}\label{cohomotopy}
 \overline{B}( \mathcal{U}_n ; R) / S( \mathcal{U}_n ; R)  \to ( \mathcal{V}_n/  \overline{B}(\mathcal{V}_n, \varepsilon_n)^c) \wedge (N_n/ L_n) 
 \end{align}
 is well-defined. 
 \end{prop} 
 \begin{proof}
 In order to prove this lemma, we follow a method used by Manolescu and Khandhawit. 
 We will use \cite[Lemma~A.4]{Khan15} (\cite[Theorem~4]{Man03}).  Set 
 \[
 \wt{K}_1(R,n) := \overline{B}( \mathcal{U}_n ; R) \cap \mathcal{F}_n^{-1}(\overset{\circ}{B}(\mathcal{V}_n, \varepsilon_n)^c)
 \]
 and  
  \[
  \wt{K}_2(R,n) := S( \mathcal{U}_n ; R) \cap \mathcal{F}_n^{-1}(\overset{\circ}{B}(\mathcal{V}_n, \varepsilon_n)^c) 
 \]
 for a fixed $R$.
  For these compact sets, we will prove that for a sufficiently large $n$, there is an isolating neighborhood $A_n= \overline{B}
  (R'; V^{\lambda_n}_{-\lambda_n} )$ satisfying the following conditions for some constant $R'$ which is independent of $n$: 
  \begin{itemize}
  \item[(a)] If $x \in \wt{K}_1(R,n)$ satisfies $t \cdot p^{\lambda_n}_{-\infty} \circ r (x) \in A_n$ for all $t \geq 0$, then for any $t \geq 0$, 
  \[
  t \cdot p^{\lambda_n}_{-\infty} \circ r (x) \notin  \partial A_n
  \]
  holds. 
  \item [(b)] If $x \in  \wt{K}_2(R,n)$, then there exists $t \geq 0$ such that $t \cdot p^{\lambda_n}_{-\infty} \circ r (x) \notin A_n$.
  \end{itemize} 
  If such an isolating neighborhood $A_n$ exists, we can apply \cite[Lemma A.4]{Khan15} and obtain an index pair $(N_n, L_n)$ such that 
  \[
 (\mathcal{F}_n (\wt{K}_1), \mathcal{F}_n (\wt{K}_2)  ) \subset  (N_n, L_n)  \text{ and } N_n \subset A_n. 
 \]
 This index pair gives our conclusion. The proof of the existence of $(N_n, L_n)$ is also similar to the proof given in \cite[Lemma A.4]{Khan15}. We give a sketch of proof. 
 We take a constant $R'$ as the constant appeared in \cref{bounded}. We also can suppose that $\overset{\circ}{B}(R') \subset V (\partial N_+)$ contains the critical set of the flow $l+c$. 
 We prove (a) by contraposition. 
  Suppose that there exists a sequence $x_n \in \wt{K}_1(R,n)$ satisfies $t \cdot p^{\lambda_n}_{-\infty} \circ r (x_n) \in A_n$ for all $t \geq 0$ and for some $t_n \geq 0$, 
  \[
  t_n \cdot p^{\lambda_n}_{-\infty} \circ r (x_n) \in  \partial A_n. 
  \]
  Here we take $R$ as a positive number with $R>R'$.
  Then, by the use of \cref{fin app conv} which we prove later, after taking a subsequence, one can assume that 
 $\{ (x_n, y_n)\}$ converges to an element 
 \[
(x, y ) \in  \cU_{k, \alpha}(N^+)\times  ( L^2_{k} (i\Lambda^1(Y\times \R^{\leq 0} )  \oplus L^2_k (S^+_{\R^{\leq 0}\times Y}  ) )
\]
 satisfying 
 \begin{itemize}
\item[(i)] the element $x+ (A_0, \Phi_0)$ is a solution of \eqref{SW eq}, 
\item[(ii)]the element $y$ is a solution of Seiberg-Witten equation on $\R^{\leq 0}\times Y$,
\item[(iii)] $y$ is temporal gauge, $d^{*_\alpha}b(t)=0$ for each $t$, where $y(t)= (b(t), \psi(t))$ and $y$ is finite type and 
\item[(iv)] $x|_{\partial N^+} = y (0)$. 
\end{itemize}
Thus, we use \cref{bounded} and obtain bounds 
  \[
  \| x\|_{L^2_{k, \alpha} } < R'  \text{ and } \|y(t)\|_{L^2_{k-\frac{1}{2}}} < R' (t  \leq 0). 
  \]
On the other hand, since we can suppose 
\[
\lim_{n\to \infty}  t_n = t_\infty \in [0, \infty ) \cup \{\infty\},
\]
 we conclude 
\[
\|y(t_\infty) \|_{L^2_{k-1}(Y)} =R'. 
\]
However, this contradicts to the choice of $R'$. The proof of (b) is similar to (a). For more detail, see \cite[Proposition 4.5]{Khan15}.

 \end{proof}
 In order to complete the proof of \cref{cohomotopy1}, we need to prove \cref{fin app conv} used in the proof of \cref{cohomotopy1}. 
\begin{prop}\label{fin app conv}
Let $\{x_n\}$ be a bounded sequence in $\mathcal{U}_{k, \alpha}$ such that 
\[
(L_{N^+}(x_n), p^{\lambda_n }_{-\infty} \circ r (x_n ) )\in \mathcal{V}_n \times V^{\lambda_n}_{-\lambda_n} 
\]
and 
\[
P_n (L_{N^+}+C_{N^+}) x_n \to 0 . 
\]
Let $y_n: [0, \infty) \to V^{\lambda_n}_{-\lambda_n}$ be a uniformly bounded sequence of trajectories such that 
\[
y_n(0) = p^{\lambda_n}_{-\infty} \circ r (x_n) . 
\]
Then, after taking a subsequence, $\{x_n\}$ converges a solution $x \in \mathcal{U}_{k, \alpha}$ (in the topology of $\mathcal{U}_{k, \alpha}$) and $\{y_n(t)\}$ converges $y(t) (\forall t\in [0,\infty))$ in $L^2_{k-\frac{1}{2}}$ which is a solution of the Seiberg-Witten equation on $\R^{\leq 0}\times Y$. 
\end{prop}
  \begin{proof}
  By a similar discussion in Proposition 3 in \cite{Man03}, we can prove that for any compact set $I \subset  (0,\infty)$ and $y_n(t)$ uniformly converges to $y(t)$ in $L^2_{k-\frac{1}{2}}$ on $I$. 
  However, for a compact set in $[0,\infty)$, we can only say $y_n(t)$ uniformly converges to $y(t)$ in $L^2_{k-\frac{3}{2}}$. In order to improve this, we will prove the following two lemmas. 
  \begin{lem}\label{claim1}
   In $L^2_{k-\frac{1}{2}}$, we have 
  \[
  p^0_{-\infty} y_n(0) \to p^0_{-\infty} y(0). 
  \]
  \end{lem}
  The proof of this convergence is completely the same as the original proof in \cite{Man03}. So we omit the proof.

  Since $\{x_n\}$ is bounded, we know that $\{x_n\}$ has a weak convergent subseqeunce and a limit $x\in L^2_{k, \alpha}$. 

\begin{lem}\label{claim2}
We have the following convergences: 
\begin{enumerate}
\item
\[
p^0_{-\infty} y_n(0) \to p^0_{-\infty} r(x) \text{ in } L^2_{k-\frac{1}{2}},
\]
\item
\[
x_n \to x \text{ in } L^2_{k, \alpha}
\]
\item
\[
y_n(0) \to y(0) \text{ in } L^2_{k-\frac{1}{2}}. 
\]
\end{enumerate}
 \end{lem}
 {\bf Proof of \cref{claim2}}
 This is also similar to \cite{Man03}, \cite{Khan15}, however, we need to use some properties of our weighted Sobolev spaces in order to deal with cone-like ends.
 \par
 (1):
 Since $V^0_{-\infty}\subset V^{\lambda_n}_{-\infty}$, the assumption $p^{\lambda_n}_{-\infty}r(x_n)=y_n(0)$ implies $p^0_{-\infty}r(x_n)=p^0_{-\infty}y_n(0)$.
 
Since $x_n$ weakly converges to $x$ in $L^2_{k, \alpha}$, $p^0_{-\infty}r(x_n)$ weakly converges to $p^0_{-\infty}r(x)$ in $L^2_{k-1/2}$.
Thus, we have $p^0_{-\infty}r(x)=p^0_{-\infty}y(0)$ and then \cref{claim1} implies (1).
\par
(2):
Since $L_{N^+}+p^0_{-\infty}\circ r: \cU_{k, \alpha} \to L^2_{k-1, \alpha}( i\Lambda^+ _{N^+} \oplus S^-_{N^+})\oplus V^0_{-\infty}(\partial N^+)$ is Fredholm, there exists a constant $C$ such that for any $x \in \cU_{k,\al }  $, 
\[
\|x\|_{L^2_{k, \alpha}}\leq C(\|L_{N^+}x\|_{L^2_{k-1, \alpha}}+\|p^0_{-\infty}\circ r (x)\|_{L^2_{k-1/2}}+\|x\|_{L^2})
\]
holds.
 By \cref{multiplication and compact embedding},  $C_{N^+}(x_n)$ converges to $C_{N^+} (x)$ in $L^2_{k-1, \alpha}$.
Thus, we have
\begin{align*}
& \|x_n-x\|_{L^2_{k, \alpha}}  \\
&\leq C(\|L_{N^+}(x_n-x)\|_{L^2_{k-1, \alpha}}+\|p^0_{-\infty}\circ r (x_n-x)\|_{L^2_{k-1/2}}+\|x_n-x\|_{L^2})\\
&\leq C(\|(L_{N^+}+C_{N^+})(x_n)-(L_{N^+}+C_{N^+})(x)\|_{L^2_{k-1, \alpha}}+\|C_{N^+}(x)-C_{N^+}(x_n)\|_{L^2_{k-1, \alpha}}\\
&+\|p^0_{-\infty}\circ y_n(0)-p^0_{-\infty}\circ r(x)\|_{L^2_{k-1/2}}+\|x_n-x\|_{L^2})\\
\end{align*}
We claim that all of  the four terms in the last line converge to zero by taking subsequences.
The first term  converges to zero since
\begin{align*}
& \|(L_{N^+}+C_{N^+})(x_n)\|_{L^2_{k-1, \alpha}}\\
& \leq \|(L_{N^+}+P_nC_{N^+})(x_n)\|_{L^2_{k-1, \alpha}}+\|(1-P_n)C_{N^+} (x_n)\|_{L^2_{k-1, \alpha}}\to 0
\end{align*}
by the assumption $(L_{N^+}+P_nC_{N^+})(x_n)\to 0$ and our choice of $P_n$ and therefore we have
\[
(L_{N^+}+C_{N^+})(x_n)\to (L_{N^+}+C_{N^+})(x)=  0 \quad \text{in }L^2_{k-1, \alpha}.
\]
The second term converges to zero
because the product estimate and the compact embedding result in \cref{multiplication and compact embedding} implies that $C_{N^+}(x_n)$ converges to $C_{N^+}(x)$ in $L^2_{k-1, \alpha}$ by taking subsequences.
The third term converges to zero by (1). 
The fourth term converges to zero since
the compact embedding result in \cref{multiplication and compact embedding} implies we can assume that $x_n$ converges to $x$ in $L^2$ by taking a subsequence.
\par
(3):
(2) implies $r(x_n)$ converges to $r(x)$ in $L^2_{k-1/2}$.
Thus we have
\begin{align*}
\|y_n(0)-r(x)\|_{L^2_{k-1/2}}
&=\|p^{\lambda_n}_{-\infty}r(x_n)-r(x)\|_{L^2_{k-1/2}}\\
&\leq \|p^{\lambda_n}_{-\infty}(r(x_n)-r(x))\|_{L^2_{k-1/2}}+\|(1-p^{\lambda_n}_{-\infty})r(x)\|_{L^2_{k-1/2}}\\
&\to 0
\end{align*}
as $n \to \infty$.
Since $y_n(0)$ converges to $y(0)$ in $L^2_{k-3/2}$, we have $r(x)=y(0)$ and thus (3) holds.
  \end{proof}
 Thus, we obtain a family of the continuous maps \eqref{cohomotopy}. Now, by the definition of Fredholm index,  we have 
 \[
\ind_\R(L_{N^+}\oplus p^{\lambda_n}_{-\infty} \circ r)= \dim_\R \mathcal{U}_n - \dim_\R \mathcal{V}_n - \dim_\R V^{\lambda_n}_{-\lambda_n} . 
 \]
By \cref{Fredholmind}, we also know
\[
\ind_\R (L_{N^+}\oplus p^{\lambda_n}_{-\infty} \circ r)
=\ind_\R (L_{N^+}\oplus p^{0}_{-\infty} \circ r)-\dim_\R V^{\lambda_n}_0
\]
\[
=-d_3(Y, [\xi])-\frac{1}{2}+2n(-Y, g_Y, \s)-\dim_\R V^{\lambda_n}_0
\]
Thus we obtain
\[
 \dim_\R \mathcal{U}_n - \dim_\R \mathcal{V}_n - \dim_\R V^{\lambda_n}_{-\lambda_n}=-d_3(Y, [\xi])-\frac{1}{2}+2n(-Y, g_Y, \s)- \dim_\R V^{\lambda_n}_{0}
\]

Then, by applying the formal (de)suspension 
\[
\Sigma^{(\frac{1}{2}-d_3(-Y, [\xi])-2n(-Y, g_Y, \s))\R \oplus (-V^0_{-\lambda_n} ) \oplus (-\mathcal{V}_n) }
\]
 to 
\[
\overline{B}( \mathcal{U}_n ; R) / S( \mathcal{U}_n ; R)  \to ( \mathcal{V}_n/  \overset{\circ}{B}(\mathcal{V}_n, \varepsilon_n)^c) \wedge (N_n/ L_n) , 
\]
we obtain a map stably written by 
\[
\Psi (Y, \xi) : S^0 \to \Sigma^{(\frac{1}{2}-d_3(-Y, [\xi])-2n(Y, g_Y, \s)) \R \oplus (-V^0_{-\lambda_n} ) }(N_n/ L_n). 
\]
We check that the domain of $\Psi (Y, \xi)$ is $S^0$.
Note that the index formula 
\[
 \dim_\R \mathcal{U}_n - \dim_\R \mathcal{V}_n - \dim_\R V^{\lambda_n}_{-\lambda_n}=-d_3(Y, [\xi])-\frac{1}{2}+2n(-Y, g_Y, \s)- \dim_\R V^{\lambda_n}_{0}
\]
implies
\[
 \dim_\R \mathcal{U}_n - \dim_\R \mathcal{V}_n - \dim_\R V^{0}_{-\lambda_n}+d_3(Y, [\xi])+\frac{1}{2}-2n(-Y, g_Y, \s)=0
\]
and thus
\[
\Sigma^{(\frac{1}{2}-d_3(-Y, [\xi])-2n(-Y, g_Y, \s))\R \oplus (-V^0_{-\lambda_n} ) \oplus (-\mathcal{V}_n) } \overline{B}( \mathcal{U}_n ; R) / S( \mathcal{U}_n ; R)=S^0.
\]
Using the definition $SWF(-Y, \s) =  \Sigma^{-2n(-Y, \s, g) \R -V^0_\lambda}  I^\mu_\lambda$, and the fact that $d_3(-Y, [\xi])=-d_3(Y, [\xi])$,  we regard $\Psi (Y, \xi)$ as a map 
\[
\Psi (Y, \xi) : S^0 \to \Sigma^{ (\frac{1}{2}-d_3(-Y, [\xi]))\R}SWF(-Y, \s). 
\]
\begin{defn}
Finally, we have 
\begin{align}\label{homotopy}
\Psi (Y, \xi) : S^0 \to \Sigma^{  (\frac{1}{2}-d_3(-Y,[\xi] )) \R  }SWF(-Y, \s). 
\end{align}
The map \eqref{homotopy} is called {\it Seiberg-Witten Floer homotopy invariant} of $(Y, \xi)$. 
\end{defn}
\begin{prop}\label{invariance}If we take a weight $\alpha$ satisfying $0<\alpha<\alpha_1$, 
then the stable homotopy class of $\Psi (Y, \xi)$ depends only on $(Y, \xi)$, where the constant $\al_1$ is introduced in \cref{bounded}. 
\end{prop}
\begin{proof}
We choose two contact forms $\theta_0$ and $\theta_1$ of $\xi$ and two complex structures $J_0$ and $J_1$ of $\xi$. Then we take 1-parameter families $\theta_t$ and $J_t$ connecting $\theta_0$ and $\theta_1$ and $J_0$ and $J_1$. Then, for such a $1$-parameter family $\theta_t$, one can take 
a positive number $\alpha_1 (\theta_t )>0$ such that there exists a 1-parameter family of finite-dimensional approximations satisfying a family version of \cref{cohomotopy1}. Moreover, we can take such a 1-parameter family of finite-dimensional approximations which are independent of $t$. This gives a homotopy between $\Psi (Y, \theta_0, J_0)$ and $\Psi (Y, \theta_1, J_1)$. The proof of independence of $\Psi (Y, \xi)$ with respect to $\al$ is essentially the same. 
This gives our conclusion. 
\end{proof}
Note that \eqref{homotopy} is not an $S^1$-equivariant map. By using the duality map $\eta$, we often regard \eqref{homotopy} as 
\[
\Sigma^{-\frac{1}{2} - d_3(Y, [\xi]) } SWF(Y) \xrightarrow{\Psi (Y, \xi) \wedge \id } 
\]
\[
\Sigma^{ \frac{1}{2}-d_3(-Y, [\xi])}SWF(-Y, \s) \wedge \Sigma^{-\frac{1}{2} - d_3(Y, [\xi]) } SWF(Y) \xrightarrow{\eta} S^0 . 
\]
We write this composition by 
\[
\widecheck{\Psi} (Y, \xi) : \Sigma^{-\frac{1}{2} - d_3(Y, [\xi]) } SWF(Y) \to S^0. 
\]

\begin{ex}
 As a trivial example, we consider homotopy classes of contact structures on $S^3$ which are parametrized by its $d_3$-invariants
 \[
 d_3(S^3, \xi_k ) = k+ \frac{1}{2} . 
 \]
  The standard contact structure is represented by  $\xi_{-1} = \xi_{std}$. (In \cite{KM07}, the homotopy class of $\xi_{std}$
  is written by $\xi_+$.) Since $SWF(-S^3) = S^0$, we have a map 
  \[
  \Psi (S^3, \xi_k ) : S^0 \to S^{k+1} . 
  \]
  Therefore, we can regard $\Psi (S^3, \xi_k )$ as an element 
  \[
   \Psi (S^3, \xi_k ) \in \pi_{-k-1}^S, 
   \]
   where $\pi_{-k-1}^S$ is the stable homotopy group of the sphere. Therefore, if $\pi_{k+1}^S= 0$, then $\Psi (S^3, \xi_k )$ must vanish. 
  \end{ex}

\section{Gluing result}\label{Gluing result}
In this section, we will prove \cref{gluing}. The main idea which we use is contained in \cite{Man07} and \cite{KLS'18}. 
In particular, we follow the arguments given in  \cite{KLS'18}. First, we introduce notions which are used in the statement of \cref{gluing}.

Let $Y$ be a rational homology 3-sphere equipped with a contact structure $\xi$ and $X$ a compact oriented 4-manifold with  $b_1(X)=0$ and $\partial X=Y$.
Suppose that a relative $Spin^c$ structure $\s_{X, \xi}=(\s_X, \s_X|_{Y}\to \s_\xi) \in \text{Spin}^c(X, \xi)$ in the sense of \cite{KM97} is given.
Now,  the relative Bauer-Furuta invariant of $(X, \s)$ is an $S^1$-equivariant stable map 
\[
\Psi(X, \s_X): (\R^{-b^+(X)}\oplus \C^{\frac{c^2_1(\s_X)-\sigma(X)}{8}})^+\to SWF(Y, \s_\xi).
\]
If we forget the $S^1$-action, this map can be written as
\[
\Psi(X, \s_X): (\R^{-b^+(X)+\frac{c^2_1(\s_X)-\sigma(X)}{4}})^+\to SWF(Y, \s_\xi)
\]
or equivalently
\[
\Psi(X, \s_X): (\R^{1+\langle e(S^+_X, \Psi_\xi), [X, \partial X]\rangle})^+ \to \Sigma^{\frac{1}{2}-d_3(Y, [\xi])}SWF(Y, \s_\xi)
\]
since
\[
\begin{split}
d_3(Y, [\xi])
&=\frac{1}{4}(c^2_1(\s_X)-2\chi(X)-3\sigma(X))-\langle e(S^+_X, \Psi_\xi), [X, \partial X]\rangle\\
&=\frac{c^2_1(\s_\omega)-\sigma(X)}{4}-\frac{\chi(X)+\sigma(X)}{2}-\langle e(S^+_X, \Psi_\xi), [X, \partial X]\rangle\\
&=\frac{c^2_1(\s_\omega)-\sigma(X)}{4}-b^+(X)-\frac{1}{2}-\langle e(S^+_X, \Psi_\xi), [X, \partial X]\rangle, 
\end{split}
\]
where $\Psi_\xi$ is a section of $S^+_X|_{Y}$ with $|\Psi_\xi (y)|=1$ for all $y \in Y$ such that the isomorphism class of $(\s_\xi, \Psi_\xi)$ corresponds to $\xi$ under the correspondence given in Lemma 2.3 in \cite{KM97}.
On the other hand, the invariant for  $(X, \omega)$ constructed in \cite{I19} is defined as a non-equivariant stable map
\[
\Psi(X, \xi, \s_{X, \xi}): (\R^{\langle e(S^+_X, \Psi_\xi), [X, \partial X]\rangle})^+ \to S^0.
\]
 Later, we will explain the invariant $\Psi(X, \xi, \s_{X, \xi})$ defined in our situation. 
\par
Finally, our contact invariant is a non-equivariant stable map
\[
\Psi(Y, \xi): S^0\to \Sigma^{\frac{1}{2}-d_3(-Y, [\xi])}SWF(-Y, \s_\xi).
\]
Using Manolescu's duality morphism 
\[
\eta: SWF(Y, \s_\xi)\wedge SWF(-Y, \s_\xi)\to S^0, 
\]
we have a non-equivariant stable map
\[
\eta\circ (\Psi(X, \s_X)\wedge \Psi(Y, \xi)): (\R^{1+\langle e(S^+_X, \Psi_\xi), [X, \partial X]\rangle})\to (\R)^+. 
\]
(Note that $d_3(-Y, [\xi])=-d_3(Y, [\xi])$.)
Therefore, we can ask whether $\eta\circ (\Psi(X, \s_X)\wedge \Psi(Y, \xi))$ and $\Psi(X, \xi, \s_{X, \xi}) $ are stably homotopy equivalent.
\par
The following gluing result can be shown by a similar way as Theorem 1 of \cite{Man07}.
\begin{thm}\label{gluing theorem}
In the above setting, $\eta\circ (\Psi(X, \s_X)\wedge \Psi(Y, \xi))$ and $\Psi(X, \xi, \s_{X, \xi})$ are stably homotopy equivalent as non-equivariant pointed maps.
\end{thm}

\subsection{The relative Bauer-Furuta invariant}
In this subsection, we summarize the definition of the relative Bauer-Furuta invariant $\Psi(X, \s_X)$ following \cite{Man03}, \cite{Man07}, and \cite{Khan15}. 
Let $X$ be a compact oriented Riemannian 4-manifold with $\partial X=Y$ is a rational homology 3-sphere.
Assume the collar neighborhood of $\partial X$ is isometric to the product.
Let $\s_X$ be a $Spin^c$ structure on $X$ and give $Y$ the $Spin^c$ structure $\s$ obtained by restricting $\s_X$ to $Y$.
We denote the spinor bundles of $\s_X$ by $S_X=S^+_X\oplus S^-_X$ and the spinor bundle of $\s$ by $S$. 
For simplicity, assume $b_1(X)=0$
\par
Let $\Omega^1_{CC}(X)$ be the space of 1-forms $a$ on $X$ in double Coulomb gauge.
The relative Bauer-Furuta invariant $\Psi(X, \s_X)$ arises as the finite-dimensional approximation of the Seiberg-Witten map
\begin{align}
\mathcal{F}^\lambda_X: L^2_{k}(i\Lambda^1_{X})_{CC}\oplus L^2_k(S^+_X)\to& L^2_{k-1}(i\Lambda^+_{X}\oplus S^-_X)\oplus V^\lambda_{-\infty} (Y)\\
(a, \phi)\mapsto& (d^+ a-\rho^{-1}(\phi\phi^*)_0, D^+_{A_0}\phi+\rho(a)\phi, p^\lambda_{-\infty}\circ  r(a, \phi))
\end{align}
for $\lambda \in \R$.
We will denote
\[
\cU_X=L^2_{k}(i\Lambda^1_{X})_{CC}\oplus L^2_k(S^+_X)\text{ and}\quad \cV_X=L^2_{k-1}(i\Lambda^+_{X}\oplus S^-_X). 
\]
We will also sometimes denote the map to the first two factors by $L_X+C_X$, 
where $L_X=d^++D^+_{A_0}+p^\lambda_{-\infty}r$ and $C_X$ is compact.
The finite-dimensional approximation goes as follows.
Pick an increasing sequence $\lambda_n\to \infty$ and an increasing sequence of finite-dimensional subspaces $\cV_{X, n} \subset \cV_X$ with $\pr_{\cV_{X, n}}\to 1$ pointwise.
Let 
\[
\cU_{X, n}=(L_X+p^{\lambda_n}_{-\infty}r)^{-1}(\cV_{X, n} \times V^{\lambda_n}_{-\lambda_n}) \subset \cU_X,
\]
 and 
\[
\cF_{X, n}:=    P_n \circ \cF^{\lambda_n}_X: \cU_{X, n}\to \cV_{X, n}\oplus V^{\lambda_n}_{-\lambda_n}, 
\]
where $P_n := \pr_{\cV_{X, n}} \times \pr_{V^{\lambda_n}_{-\lambda_n}} $.
Let 
\[
\wt{K}^1_{X, n}= (\cF_{X, n})^{-1}(\overline{B}(\cV_{X, n}; \varepsilon_n)\times  V^{\lambda_n}_{-\lambda_n}) \cap \overline{B}(\cU_{X, n}, R),
\]
\[
 \quad \wt{K}^2_{X, n}=(\cF_{X, n})^{-1}(\overline{B}(\cV_{X, n}; \varepsilon_n)\times  V^{\lambda_n}_{-\lambda_n})\cap S(\cU_{X, n}, R)
\]
\[
{K}^1_{X, n}=\pr_{V^{\lambda_n}_{-\lambda_n}}\circ \cF_{X, n}(\wt{K}^1_{X, n}), \quad K^2_{X, n}=\pr_{V^{\lambda_n}_{-\lambda_n}}\circ \cF_{X, n}(\wt{K}^2_{X, n})
\]
for some $R>0$.
One can find an index pair $(N_X, L_X)$ which represents the Conley index for 
$V^{\lambda_n}_{-\lambda_n}$ in the form $N_X/L_X$ such that $K^1_{X, n}\subset N_X$ and $K^2_{X, n}\subset L_X$.
\par
Now, for a sufficiently large $n$, we have a map
\[
\mathcal{F}_{X, n}: \overline{B}(\cU_{X, n}, R) /S(\cU_{X, n}, R)\to  (\cV_{X, n}/(\overline{B}(\cV_{X, n}, \varepsilon)^c)) \wedge (N_X/L_X).
\]
This gives the relative Bauer-Furuta invariant $\Psi(X, \s_X)$ constructed by Manolescu(\cite{Man03}) and Khandhawit(\cite {Khan15}).
\subsection{The Bauer-Furuta version of Kronheimer-Mrowka's invariant}
In this subsection, we summarize the definition of the Bauer-Furuta version of Kronheimer-Mrowka's invariant $\Psi(X, \xi , \s_{X, \xi})$ following \cite{I19}, though
the weighted Sobolev spaces we use here are different from those used in \cite{I19}.
\par
Let $X$ be a compact oriented 4-manifold with nonempty boundary.
We assume $H^1(X, \partial X; \R)=0$, in particular, $Y=\partial X$ is connected.
Let $\xi$ be a contact structure on $Y=\partial X$ compatible with the boundary orientation.
As in the construction of $N^+$, we will construct a complete Riemannian manifold $(X^+, g_0)$ by attaching an almost K\"ahler conical end.
As a manifold, 
\[
X^+=X\cup_Y ([0, 1]\times Y) \cup_Y [1, \infty)\times Y=X\cup_Y  N^+.
\]
Pick a contact 1-form $\theta$ on $Y$ and a complex structure $J$ of $\xi$ compatible with the orientation.
There is now a unique Riemannian metric $g_1$ on $Y$ such that $\theta$ satisfies that $|\theta|=1$, $d\theta=2*\theta$, and $J$ is an isometry for $g|_\xi$ .
\par
Define a symplectic form $\omega_0$ on $[1, \infty)\times Y$ by the formula
\begin{align}
\omega_0&=\frac{1}{2}d(s^2\theta)\\
&=sds\wedge \theta+\frac{1}{2}s^2 d\theta
\end{align}
and a metric $g_0$ by
\[
g_0=ds^2+s^2g_1.
\]
Pick a smooth extension of $g_0$ to all of $X^+$ which is a product metric on $[0, 1/2]\times Y$.
\par
On $X^+\setminus X$, we have a canonical $Spin^c$ structure $\s_0$, a canonical $Spin^c$ connection $A_0$, a canonical positive spinor $\Phi_0$ as before.
Fix a $Spin^c$ structure $\s_X $ on $X^+$ equipped with an isomorphism $\s_X \to \s_0$ on $X^+\setminus X$.
We denote such a pair by $\s_{X, \xi}$.
Fix a smooth extension of $(A_0, \Phi_0)$ such that $\Phi_0$ is zero on $X \cup ([0, 1]\times Y)$ and $A_0$ is product on $[0, 1/2]\times Y$.
We also fix a nowhere zero proper extension $\sigma$ of $s \in [1, \infty)$ coordinate to all of $X^+$ which is $1$ on $X \cup ([0, 1]\times Y)$.
(This implies that for a section supported in $X$, its weighted Sobolev norms are equal to its unweighted Sobolev norms.)

On $X^+$, weighted Sobolev spaces 
\[
\widehat{\cU}_{X^+}=L^2_{k, \alpha}(i \Lambda^1_{X^+}\oplus S^+_{X^+})
\]
\[
\widehat{\cV}_{X^+}=L^2_{k-1, \alpha}(i\Lambda^0_{X^+}\oplus i\Lambda^+_{X^+}\oplus S^-_{X^+})
\]
are defined as before using $\sigma$ for a positive real number $\al \in \R$ and $k \geq 4$, where $S^+_{X^+}$ and $S^-_{X^+}$ are positive and negative spinor bundles. 
\par
The invariant $\Psi(X, \xi , \s_{X, \xi})$(\cite{I19}) is obtained as a finite-dimensional approximation of the Seiberg-Witten map
\begin{align}\label{FX+}
\begin{split}
&\widehat{\mathcal{F}}_{X^+}:  \widehat{\cU}_{X^+}\to \widehat{\cV}_{X^+}\\
&(a, \phi)\mapsto (d^{*_\alpha}a, d^+a-\rho^{-1}(\phi\Phi^*_0+\Phi_0\phi^*)_0-\rho^{-1}(\phi\phi^*)_0, D^+_{A_0}\phi+\rho(a)\Phi_0+\rho(a)\phi)
\end{split}
\end{align}
\par
The finite-dimensional approximation goes as follows.
We decompose $\widehat{\cF}_{X^+}$ as $\widehat{L}_{X^+}+\widehat{C}_{X^+}$ where 
\[
 \widehat{L}_{X^+}(a, \phi)=(d^{*_\alpha}a, d^+a-\rho^{-1}(\phi\Phi^*_0+\Phi_0\phi^*)_0, D^+_{A_0}\phi+ \rho(a)\Phi_0) 
 \]
 and 
 \[
  \widehat{C}_{X^+}(a, \phi) = (0, -\rho^{-1}(\phi\phi^*)_0, \rho(a)\phi). 
  \]
For $0<\alpha\leq \alpha_1$, $\widehat{L}$ is Fredholm.
In this section, we fix a weight $\alpha \in (0, \infty)$ satisfying $\alpha \leq \al_1$, where $\al_1$ is the constant appeared in \cref{bounded}. 
Then $\widehat{L}_{X^+}$ is linear Fredholm and  $\widehat{C}_{X^+}$ is quadratic, compact.
\par
Then pick an increasing sequence $\lambda_n\to \infty$ and an increasing sequence of finite-dimensional subspaces $\widehat{\cV}_{X^+, n} \subset \cV_{X^+}$ such that $\pr_{\widehat{\cV}_{X^+, n}} \to 1$ pointwise.
Let 
\[
\widehat{\cU}_{X^+, n}=\wh{L}^{-1}(\widehat{\cV}_{X^+, n})\subset \wh{\cU}_{X^+},
\]
 and 
\[
\mathcal{F}_{X^+, n}:=  \pr_{\widehat{\cV}_{X^+, n}}   \circ \mathcal{F}_{X^+}: \widehat{\cU}_{X^+, n}\to \widehat{\cV}_{X^+, n}. 
\]
We can show that for a large $R>0$, a small $\varepsilon$ and a large $n$, we have a well-defined map
\[
\mathcal{F}_{X^+, n}: B(\widehat{\cU}_{X^+, n}, R)/S(\widehat{\cU}_{X^+, n}, R)\to B(\widehat{\cV}_{X^+, n}, \varepsilon)/S(\widehat{\cV}_{X^+, n}, \varepsilon).
\]
This gives the Bauer-Furuta version of Kronheimer-Mrowka's invariant 
\[
\Psi(X, \xi , \s_{X, \xi}) \in \pi^S_{\langle e(S^+_X, \Phi_0), [(X, \partial X)]\rangle} 
\]
 defined in \cite{I19}. The following result is proved in \cite{I19}: 
\begin{thm}
For $\alpha \in (0, \infty)$ satisfying $\alpha \leq \al_1$,  the stable homotopy class of $\Psi(X, \xi , \s_{X, \xi}) \in \pi^S_{\langle e(S^+_X, \Phi_0), [(X, \partial X)]\rangle} $ depends only on $(X, \xi , \s_{X, \xi})$, where $\pi_i^S$ be the $i$-th stable homotopy group of the sphere. Moreover, in the case of   
\[
\langle e(S^+_X, \Phi_0), [(X, \partial X)]\rangle=0,
\]
 the mapping degree of $\Psi(X, \xi , \s_{X, \xi})$ recovers the Kronheimer-Mrowka's invariant of $(X, \xi , \s_{X, \xi})$ up to sign. 
\end{thm}


\subsection{Deforming the duality pairing}
First, we deform the duality pairing. We consider a counterpart of \cite[Proposition~6.10]{KLS'18}.  Although in the situation of \cite{KLS'18}, $X_0$ and $ X_1$ are compact, these facts are not essential in the proof of \cite[Proposition~6.10]{KLS'18}. Therefore, in our situation, we have a similar result: 
\begin{prop}
The morphism $\eta \circ (\Psi (X, \xi, \s_{X, \xi} ) \wedge \Psi(Y, \xi) )$ can be represented by a suitable desuspension of the map 
\begin{align}\label{final1}
\frac{\overline{B}(\cU_{X, n} , R_1)}{S(\cU_{X, n} , R_1)} \wedge \frac{\overline{B}(\cU_{N^+, n} , R_2)}{S(\cU_{N^+, n} , R_2)} \to
\frac{\overline{B}(\cV_{X, n} , \varepsilon)}{S(\cV_{X, n} , {\varepsilon} )}  \wedge \frac{\overline{B}(\cV_{N^+, n} , {\varepsilon})}{S(\cV_{N^+, n} , {\varepsilon} )}\wedge \frac{\overline{B}(V_{-\lambda_n}^{\lambda_n} ,  \overline{\varepsilon})}{S(V_{-\lambda_n}^{\lambda_n}  ,  \overline{\varepsilon} )} 
\end{align} 
\[
(x_1, x_2 ) \mapsto
\]
\[
 \begin{cases} 
(\mathcal{F}_{X, n} (x_1),\mathcal{F}_{N^+, n} (x_2), r x_1 - r x_2) & \text{  if  } \| \mathcal{F}_{X, n} (x_1) \| \leq \varepsilon \text{ and }\| \mathcal{F}_{N^+, n} (x_2) \| \leq \varepsilon \\ 
* & \text{  otherwise  }
\end{cases}
\]
for large numbers $R_1$, $R_2$ and small positive numbers $\varepsilon$, $\overline{\varepsilon} $, where the maps $r$ are coming from the restrictions. 
\end{prop}

\subsection{Proof of the gluing theorem}
In this subsection, we give a proof of \cref{gluing}. 
We follow the methods given by Khandhawit-Sasahira-Lin  \cite{KLS'18}. We use the following notations: 
 \begin{itemize}
 \item $\widehat{\cU}_{X}=L^2_k( i\Lambda^1\oplus S^+_X)$, $\cU_{X}=L^2_k( i\Lambda^1\oplus S^+_X)_{CC}$; 
 \item $\widehat{\cU}_{N^+}=L^2_{k, \al} ( i\Lambda^1\oplus S^+_{N^+})$, ${\cU}_{N^+}=L^2_{k, \al} ( i\Lambda^1\oplus S^+_{N^+})_{CC}$;
 \item $\widehat{\cU}_{X^+}=L^2_{k, \al} (X^+; i\Lambda^1\oplus S^+_{X^+})$, ${\cU}_{X^+}=i\ker d^{*_\alpha}\oplus L^2_{k, \al} (X^+; S^+_{X^+}) \subset L^2_{k, \al} (X^+; i\Lambda^1\oplus S^+_{X^+} )$; 
 \item $ \widehat{\cV}_{X}=L^2_{k-1}( i\Lambda^0_X\oplus i\Lambda^+_X\oplus S^-_{X})$, ${\cV}_{X}=L^2_{k-1}( i\Lambda^+_X \oplus S^-_X)$; 
\item $ \widehat{\cV}_{N^+}=L^2_{k-1, \alpha}( i\Lambda^0_{N^+} \oplus i\Lambda^+_{N^+} \oplus S^-_{N^+})$, ${\cV}_{N^+}=L^2_{k-1, \alpha}(  i\Lambda^+_{N^+} \oplus S^-_{N^+} )$
 \end{itemize}
 for a real number $\al \in \R$ and $k \geq 4$. In this subsection, we fix a weight $\alpha \in (0, \infty)$ satisfying $\alpha \leq \al_1$, where $\al_1$ is the constant appeared in \cref{bounded}.  
 Before proving the gluing theorem, we introduce a notion of BF pair which is a counterpart of SWC triple in \cite{KLS'18}. 
 Let $H_1$, $H_2$ be separable Hilbert spaces. 
 \begin{defn} Let $(L, C)$ be a pair of bounded continuous maps from $H_1$ to $H_2$. Suppose $L$ is a Fredholm linear map and $C$ extends to a continuous map $\overline{H}_1 \to H_2$, where $\overline{H}_1$ is a completion of $H_1$ with respect to a weaker norm. We impose that $H_1 \to \overline{H}_1$ is compact.  Moreover, we assume 
 \[
 (L+C)^{-1}(0) \subset \overset{\circ}{B}(H_1, M') 
 \]
 for some $M'>0$.
 Then $(L, C)$ is called a {\it BF pair}. 
\end{defn}
As in the case of SWC triples, we also have a notion of c-homotopic. 
\begin{defn}Two BF pairs $(L_i, C_i) (i=1, 2)$ are {\it c-homotopic} if there is a homotopy between them through a continuous family of BF pairs with a uniform constant $M'$.
 Two BF pairs $(L_i, C_i) (i=1, 2)$ are {\it stably c-homotopic} if there exist Hilbert spaces $H_3$, $H_4$ such that 
$(L_1\oplus  id_{H_3} , C_1 \oplus 0 )$ is c-homotopic to $(L_2\oplus id_{H_4} , C_2 \oplus 0)$. 
\end{defn}
Similar to the case of SWC triples, for a given BF pair $(L, C)$, we can define a stable cohomotopy invariant
\[
\Psi(L, C) \in \{ S^{\ind (L)} , S^0 \}. 
\]

In the proof, we have seven steps as in \cite{KLS'18}. 
\begin{itemize}
\item[Step 1] In \cite{KLS'18}, we move the gauge fixing condition $d^{*_\alpha}=0$ from the domain to the maps. In our case, we do not need to do anything because \eqref{FX+} contains $d^{*_\alpha}$ as a component. 
\item[Step 2] We glue the Sobolev spaces of the domains. 
\item[Step 3, 4]We glue the Sobolev spaces of the targets. 
\item[Step 5] We focus on deforming the boundary conditions for gauge fixing. 
\item[Step 6] We change the action of harmonic gauge transformations with different boundary conditions. However, in our case, these symmetries are trivial. Moreover, we recover double Coulomb gauge conditions. 
\item[Step 7] We make the final homotopy between \eqref{final1} and  \eqref{fiber5}. 
\end{itemize}
We do not need Step 1, so we start with Step 2. 
\subsubsection{Step 2}  
We can prove the following lemma of gluing Sobolev spaces by the same argument as Lemma 3 in \cite{Man07}.  Since the proof is essentially the same, we omit the proof. 
\begin{lem}\label{gluing Sobolev}
Regard $X^+=X\cup_{\{0\}\times Y} N^+$. 
We can assume that $X$ also has cylindrical end near the boundary, and denote by $s$ the variable in the direction normal to $\{0\}\times Y$.
Let $E$ be an admissible vector bundle over $X^+$ and assume that the $L^2_{k, \al}$-Sobolev completion of the space of smooth, compactly supported sections of $E$ on $X^+$, $N^+$ are defined by a fixed connection and a fixed pointwise norm.
Then, for  $k\geq \Z^{\geq 1}$ and $\al \in \R$, there is a natural identification
\[
L^2_{k, \al}(X^+; E)=L^2_k(X; E)\times_{\prod^{k-1}_{m=0}L^2_{m+1/2}(Y; E)}L^2_{k, \al}(N^+; E), 
\]
where the right-hand side is the fiber product of $L^2_k(X; E)$ and $L^2_{k, \al}(N^+; E)$
with respect to the maps
\[
r^{k}_1: L^2_k(X)\to \prod^{k-1}_{m=0}L^2_{k-\frac{1}{2} - m}(Y; E),
\]
\[
r^k_1(u)=\left(u|_{Y}, \frac{\partial u}{\partial s}|_{Y}, \frac{\partial^2 u}{\partial s^2}|_{Y},  \cdots, \frac{\partial^{k-1} u}{\partial s^{k-1}}|_{Y}\right),
\]
\[
r^k_2: L^2_{k, \al}(N^+; E)\to \prod^{k-1}_{m=0}L^2_{k-\frac{1}{2} - m}(Y; E),
\]
\[
r^k_2(u)=\left(u|_{Y}, \frac{\partial u}{\partial s}|_{Y}, \frac{\partial^2 u}{\partial s^2}|_{Y},  \cdots, \frac{\partial^{k-1} u}{\partial s^{k-1}}|_{Y}\right).
\]
\qed
\end{lem}
By the use of \cref{gluing Sobolev}, we glue configurations. 
Before gluing, we introduce a family of linear maps: 
\[
D^{(\leq k)} :  \wh{\cU}_{N^+}\times \wh{\cU}_{X}  \to \bigoplus_{m=0}^k \widehat{V}_{k-m-\frac{1}{2}} 
\]
defined by 
\[
D^{(\leq k)}  (x_1, x_2 ) := r^k_1(x_1)- r^k_2(x_2)
\]
for any non-negative integer $k$. The following statement is followed by \cref{gluing Sobolev}. 
\begin{lem}For any $k\in \Z_{\geq 0}$,  
the map $D^{(\leq k)}$ is surjective and the kernel can be identified with $\widehat{\cU}_{X^+}$. 
 \end{lem}
 Now, we glue the configuration spaces. 
\begin{lem}\label{fiber prod}The pair 
\begin{align}\label{fiber1}
( ( \operatorname{pj} (\wh{L}_{X} \times \wh{L}_{N^+}) , D^{(\leq k)} ) , \operatorname{pj}(\wh{C}_{X} \times \wh{C}_{ N^+ })  ) 
\end{align}
 is a BF pair from $ \wh{\cU}_{N^+} \times \wh{\cU}_{X}  $ to $ \wh{\cV}_{X^+} \times   \bigoplus_{m=0}^k \widehat{V}_{k-m} $, where $\operatorname{pj}$ is the projection from $\wh{\cV}_{N^+} \times \wh{\cV}_{X}$ to $\wh{\cV}_{X^+}$. (Here we regard $\wh{\cV}_{X^+}$ as the kernel of $D^{(\leq m)}$. ) 
 
   Moreover, $( (\operatorname{pj} (\wh{L}_{X} \times \wh{L}_{N^+}) , D^{(\leq k)} ) , \operatorname{pj}(\wh{C}_{X} \times \wh{C}_{ N^+ })  ) $ is stably c-homotopic to $( \wh{L}_{X^+}, \wh{C}_{X^+} )$. 
\end{lem}
\begin{proof}
In the proof, we use the following lemma. This is an easier version of Lemma 6.13 in \cite{KLS'18} and originally proved in Observation 1 in \cite{Man07}. 
\begin{lem}\label{snake} 
Let $(L, C)$ be pair of continuous maps from $H_1$ to $H_2$. Suppose $L$ is bounded linear and $C$ extends to $\overline{H}_1\to H_2$ for a weak norm of $H_1$. Let $g$ be a surjective linear map $H_1 \to H_3$.
Then the following conditions are equivalent; 
\begin{itemize}
\item $(L \oplus g, C\oplus 0 )$ is a BF pair, and 
\item $(L|_{\ker g} , C|_{\ker g} )$ is a BF pair. 
\end{itemize}
Moreover,  $(L \oplus g, C\oplus 0 )$ is c-homotopic to $(L|_{\ker g} , C|_{\ker g} )$. 
\end{lem}
\cref{fiber prod} is followed by using \cref{snake}.

\end{proof}

\subsubsection{Step 3, 4} 
For any positive integer $k$, we define 
\[
E^{(\leq k-1)} : \wh{\cV}_{X} \times \wh{\cV}_{N^+ }\to \bigoplus_{m=0}^{k-1} \widehat{V}_{k-m-\frac{1}{2}}
\] 
by 
\[
E^{(\leq k-1)} (y_1, y_2 )= r^{k-1}_1(y_1) - r^{k-1}_2(y_2 ). 
\]
The following lemma is a counterpart of Proposition 6.17 in \cite{KLS'18}. 
\begin{lem}The pair
\begin{align}\label{fiber2} \begin{split}
(  ( \operatorname{pj} \circ (\wh{L}_{X} \times \wh{L}_{N^+}), E^{(\leq k-1)}\circ (\wh{L}_{X} \times \wh{L}_{N^+}), D^{(\leq 0)}), 
\\
(   \operatorname{pj} \circ (\wh{C}_{X} \times \wh{C}_{N^+}), E^{(\leq k-1)}\circ (\wh{C}_{X} \times \wh{C}_{N^+}) , 0 )  )
\end{split}
\end{align} 
is a BF pair from $\wh{\cU}_{X} \times \wh{\cU}_{N^+ } $ to $\wh{\cV}_{X^+} \times (\bigoplus_{m=0}^{k-1} \widehat{V}_{k-m-\frac{1}{2}} ) \times  \widehat{V}_{k-\frac{1}{2}} $. Moreover, this BF pair is stably c-homotopic to \eqref{fiber1}. 

 \end{lem}
The proof is essentially the same as \cite[Proposition~6.17]{KLS'18}. Thus, we omit the proof.

The following lemma is a counterpart of Lemma 6.19 in \cite{KLS'18}. The only difference is that we have no constant functions in $\wh{\cV}_{X^+}$.
 \begin{lem} \label{codomain} The map 
 \[
 ( \operatorname{pj} , E^{\leq k-1} ) : \wh{\cV}_{X} \times \wh{\cV}_{N^+ } \to \wh{\cV}_{X^+}\times  \bigoplus_{m=0}^{k-1} \widehat{V}_{k-m-\frac{1}{2}}
 \]
is an isomorphism. 
 \qed
 \end{lem}
%
Then one can prove the main result of Step 3 and 4. 
\begin{lem}\label{fiber33}
The pair in \eqref{fiber2} can be identified with the pair 
\begin{align} \label{fiber3}
( (  (\wh{L}_{X}\times \wh{L}_{N^+}) , D^{(\leq 0)}  ),  (\wh{C}_{X}\times \wh{C}_{N^+}, 0) )
\end{align}
from $\wh{\cU}_{X} \times \wh{\cU}_{N^+ } $ to $\wh{\cV}_{X} \times \wh{\cV}_{N^+ }  \times \wh{V}_{k-\frac{1}{2}} $ via the isomorphism given in \eqref{codomain}. 
\end{lem}
This is a counterpart of Lemma 6.20 in \cite{KLS'18}. This is a corollary of \cref{codomain}. 

\subsubsection{Step 5} 
This step contains the non-trivial argument which appears in our situation. 
We sometimes omit spinors from expressions in this step.
Let us consider an operator 
\[
\overline{d} : L^2_{k-\frac{1}{2}} ( i\Lambda^0_Y )_0 \to d L^2_{k+\frac{1}{2}} ( i\Lambda^0_Y )_0 
\]
defined in \cite[Step 5]{KLS'18}. We denote by $\overline{d}^*$ its formal adjoint. Then we have a family of maps 
\[
D_{H,t} : \wh{\cU}_{X}\times \wh{\cU}_{N^+ }\to d L^2_{k+\frac{1}{2}} ( i\Lambda^0_Y ) \oplus L^2_{k-\frac{1}{2}} ( i\Lambda^0_Y )
\]
given by 
\begin{align*}
&D_{H,t} ( a_1, a_2) := 
\\
& ( \operatorname{pr}_{\im d} (a_1|_Y-a_2|_Y), t \overline{d}^*(\operatorname{pr}_{\im d} (a_1|_{Y}+a_2|_Y)) + (1-t) \operatorname{pr}_{L^2_{k+\frac{1}{2}} ( i\Lambda^0_Y ) } ( a_1|_Y-a_2|_Y )
\end{align*}
parametrized  by $t \in [0,1]$. 
The next proposition is a counterpart of Proposition 6.22. 
\begin{prop}
For any $t \in [0,1]$, the pair 
\begin{align}\label{fiber44}
\left(  \left(\wh{L}_{X},  \wh{L}_{N^+} , D_Y \oplus D_{H,t} \right) , \left(\wh{C}_{X}\times \wh{C}_{N^+}, 0 \right) \right)
\end{align}
is a BF pair from $\wh{\cU}_{X} \times \wh{\cU}_{N^+ } $ to $\wh{\cV}_{X} \times \wh{\cV}_{N^+ }  \times \wh{V}_{k-\frac{1}{2}} $
In particular, 
\begin{align}\label{fiber4}
\left(  \left(\wh{L}_{X},  \wh{L}_{N^+} , D_Y \oplus D_{H,1} \right) , \left(\wh{C}_{X}\times \wh{C}_{N^+}, 0 \right) \right)
\end{align}
is stably c-homotopic to \eqref{fiber4}. 
\end{prop}
\begin{proof}
As in the proof of Proposition 6.22 in \cite{KLS'18}, we first prove the following result; 
\begin{prop}\label{Neumann}
Let $W \subset L^2_{k+1} (X; \R) \times L^2_{k+1, \al} (N^+; \R)$ be the subspace containing all functions $(f_1, f_2)$ satisfying the following conditions; 
\begin{itemize}
\item[(i)]  $\Delta f_i= 0 $
\item[(ii)]  $f_1 ( \hat {o} ) = 0$, and
\item[(iii)]  $f_1 |_Y = f_2  |_Y$, 
\end{itemize}
where $\hat {o}$ is a fixed point in $Y$.
Then the map 
$\rho_t : W \to L^2_{k-\frac{1}{2}} (  \Lambda^0_Y )_0$ defined by 
\[
\rho_t (f_1, f_2) := 2t \overline{d}^* d (f_1|_{Y}) + (1-t ) ( \partial_{\vec{n} } f_1|_Y - \partial_{\vec{n} }f_2|_Y ) 
\]
is an isomorphism, where 
\[
L^2_{k-\frac{1}{2}} (  \Lambda^0_Y )_0 := \{ f \in L^2_{k-\frac{1}{2}} (i \Lambda^0_Y ) |  \int_Y f d\operatorname{vol}_Y =0 \} . 
\]
\end{prop}
\begin{proof}[Proof of \cref{Neumann}] 
When $t=1$, we can use a similar argument in Proposition 2.2 in \cite{Khan15} since we have \cref{Slice}. 
For $t<1$, we can use the same argument given in Proposition 6.22 in \cite{KLS'18}. 
\end{proof}
When $t=0$, \eqref{fiber44} is a BF pair by \cref{fiber33}. 
For each element in the kernel of \eqref{fiber44}, there is a unique gauge transformation to an element in the kernel of 
\[
\left(  \left(\wh{L}_{X},  \wh{L}_{N^+} , D_Y \oplus D_{H,0} \right) , \left(\wh{C}_{X}\times \wh{C}_{N^+}, 0 \right) \right). 
\]
This proves that the kernel of \eqref{fiber44} is finite dimensional for any $t$. The remaining part is the same as the proof of Proposition 6.22.


\end{proof}

\subsubsection{Step 6} 
In this step, we see counterparts of Lemma 6.24 and Corollary 6.25 in \cite{KLS'18}.
\begin{lem}
The operator 
\[
( d^{*}_X, d^{*_\alpha}_{N^+} ,  D_{H, 1} ) : 
\]
\[
\wh{\cU}_{X}\oplus \wh{\cU}_{N^+} \to L^2_{k-1} ( i\Lambda^0(X)) \oplus L^2_{k-1, \alpha} ( i\Lambda^0(N^+)) 
\oplus  d L^2_{k-\frac{1}{2}} ( i\Lambda^0_Y ) \oplus L^2_{k-\frac{1}{2}}  ( i\Lambda^0_Y )_0
\]
is surjective and its kernel can be written as 
\[
(L^2_{k} ( i\Lambda^1_{X})_{CC} \oplus L^2_{k} (S^+_X ) ) \times  ( L^2_{k, \alpha} ( i\Lambda^1(X^+ ))_{CC} \oplus L^2_{k, \alpha} (S^+_{N^+}  ) ) .  
\]
\end{lem}
\begin{proof}
This is obtained by using integration by parts. 
\end{proof}
\begin{cor}
The BF pair \eqref{fiber4} is c-stable homotopic to a BF pair 
\begin{align}\label{fiber5}
 ( ( L_X, L_{N^+} , D^{(\leq 0)} ),  (C_X, C_{N^+} , 0 )  ) 
 \end{align}
 from $\cU_X \times \cU_{N^+} $ to $\cV_X \times \cV_{N^+} \times V(Y)$. 
\end{cor}
\begin{proof}
This is a corollary of \cref{snake}. 
\end{proof}

\subsubsection{Step 7} 
We choose finite dimensional vector spaces $\cV_{X, n }$ and $\cV_{N^+ , n }$ in $L^2_{k-1} (i\Lambda^+_X\oplus S^-_X)$ and $L^2_{k-1, \alpha } (i\Lambda^+_{N^+} \oplus S^-_{N^+}) $. Note that we denote by $V^{\lambda_n}_{-\lambda_n}$ a family of finite dimensional approximation of $V(Y)$.
We introduce a family of subbundles :
\begin{align*}
W^{n, t}_{X, N^+} := \{ (x_1, x_2) \in \cU_X \times \cU_{N^+}  | L_X(x_1) \in \cV_{X, n }, \ L_{N^+}(x_2) \in \cV_{N^+, n } \\ 
p^{\infty} _{\lambda_n } r_2 (x_2) = t p^{\infty} _{\lambda_n }  r_2 (x_1) \\
p^{-\lambda_n} _{-\infty } r_2 (x_1) = t p^{-\lambda_n }_{-\infty}  r_2 (x_2) \} . 
\end{align*}
This gives a vector bundle $\bigcup_{t \in [0,1] }  W^{n, t}_{X, N^+}  \to [0,1]$. 
A point is that the operators 
\[
\wh{p}^{\infty}_{0}\circ \wh{r}  : \{ x \in  \wh{\cU}_X | \wh{L}_X (x) =0 \}   \to \wh{V}^{\infty}_{0} (Y)
\]
\[
\wh{p}^{0}_{-\infty} \circ \wh{r} : \{ x \in  \wh{\cU}_X | \wh{L}_{N^+} (x) =0 \}   \to \wh{V}^{0}_{-\infty}(Y) 
\]
are compact. The first fact is proved in \cite[Theorem~17.1.3. (ii)]{KM07}. By the same proof, we can prove that the second operator is also compact. Moreover, 
\[
\wh{p}^{-\lambda_n}_{-\infty} \circ \wh{r}  : \{ x \in  \wh{\cU}_X | \wh{L}_X (x) =0 \}   \to \wh{V}^{-\lambda_n}_{-\infty}  (Y)
\]
\[
\wh{p}^{\infty}_{\lambda_n} \circ \wh{r} : \{ x \in  \wh{\cU}_X | \wh{L}_{N^+} (x) =0 \}   \to \wh{V}^{\infty}_{\lambda_n}(Y) 
\]
are surjective for a sufficient large $n$. This is a corollary of the unique continuation property. These two facts enable us to see that the rank of $W^{n, t}_{X, N^+} $ is constant. 
Finally, we see the following boundedness result which is a counterpart of \cite[Lemma~6.26]{KLS'18}. 
\begin{prop} 
For any $R>0$, there exist $N$, $\varepsilon_0$ with the following significance: For any $n>N$, $t \in [0,1], (x_1, x_2 )\in \overset{\circ}{B}(W^{n, t}_{X, N^+}, R) $ and $\gamma_i : (-\infty, 0] \to \overset{\circ}{B} (V^{\lambda_n}_{-\lambda_n} , R)$ for $ i=1, 2$ satisfying 
\begin{itemize}
\item[(i)]  $\| p^{\lambda_n}_{-\lambda_n}  (r_2 (x_1) - r_2 (x_2) ) \|_{L^2_{k-\frac{1}{2}}}  \leq \varepsilon$
\item[(ii)]  $\|\pr_{ \cV_{X, n }}  \circ (  \wh{L}_{X} + \wh{C}_{X} ) (x_1) \|_{L^2_{k-1}} \leq \varepsilon$, $\|\pr_{ \cV_{N^+, n }} \circ (  \wh{L}_{N^+} + \wh{C}_{N^+} ) (x_2) \|_{L^2_{k-1, \alpha}} \leq \varepsilon$, 
\item [(iii)] $\gamma_i$ is approximated trajectory with $\gamma_i(0) = p^{\lambda_n}_{-\lambda_n} \circ \wh{r}_i (x_i)$ for $i=1$ and $2$, 
\end{itemize}
we have $\|x_1\|_{L^2_{k+1}} \leq R+ 1$, $\|x_2\|_{L^2_{k+1, \al}} \leq R+ 1$, $\|\gamma_i (t)\|_{L^2_{k-\frac{1}{2}} } \leq R +1$ for $i=1$ and $2$.
\end{prop}
\begin{proof}
The proof is essentially identical with \cite[Lemma 1]{Man07}.  
\end{proof}
Then the restricting family on $W^{n, t}_{X, N^+}$ defines a homotopy between \eqref{final1} and  \eqref{fiber5}
This completes the proof of the gluing theorem. 

At the end of this subsection, we see the following corollary of \cref{gluing}. 
\begin{cor}\label{right}
Let $Y$ be a rational homology $3$-sphere equipped with a contact structure $\xi$. If $\xi$ has a symplectic filling with $b_1=0$, then \eqref{ourinv} has a non-equivariant stable homotopy left inverse. 
  In particular, \eqref{ourinv} is not stably null-homotopic. 
Moreover, a left inverse is given by (the dual of) the relative Bauer-Furuta invariant for the filling. 
\end{cor}
\begin{proof}
Let $(X, \om)$ be such a symplectic filling of $(Y, \xi)$.
We see the following maps: 
\begin{align}\label{zusiki}
 \begin{CD}
 S^{m  + 2n +\frac{c^2_1(\s_X)-\sigma(X)}{4}} @>{\Psi_{(X, \s_X)}}>>   \Sigma^{m+b^+(X) + 2n} SWF(Y)    \\
  @.    @V{\Psi(Y, \xi)\wedge id }VV    \\
  @.  \Sigma^{\frac{1}{2} - d_3(-Y, [\xi]) } SWF(-Y) \wedge  \Sigma^{m+b^+(X) + 2n} SWF(Y)  \\
  @. @\vert \\ 
@.   \Sigma^{\frac{1}{2} - d_3(-Y, [\xi])  + m + b^+(X) + 2n } SWF(-Y)  \wedge SWF(Y)  \\
@.    @V{ id \wedge \eta  }VV    \\
@. S^{\frac{1}{2} - d_3(-Y, [\xi])  + m + b^+(X) + 2n} \\
 \end{CD}.
\end{align}

The gluing theorem implies that 
$(id \wedge \eta) \circ  (\Psi(Y, \xi)\wedge id) \circ  {\Psi_{(X, \s_X)}}$ and $\Psi_{(X, \s_{\om}, \xi)}$ are stably homotopic. 
Since $(X, \om)$ is a symplectic filling, $\Psi_{(X, \s_{\om}, \xi)} $ is a homotopy equivalence (\cite{I19}).  
Note that by the definition of $d_3(Y, [\xi])$, we can see 
\[
m  + 2n +\frac{c^2_1(\s_X)-\sigma(X)}{4}=\frac{1}{2} - d_3(-Y, [\xi])  + m + b^+(X) + 2n.
\]
So the dimension of the spheres of the domain and the codomain are equal. 
This implies the conclusion. 

\end{proof}

\begin{rem}\cref{right} implies that, under the same assumption as \cref{right},  
the dual map 
\[
 \widecheck{\Psi} (Y, \xi) :  \Sigma^{-\frac{1}{2}- d_3(Y)}SWF(Y, \s)\to S^0 
 \]
 has a non-equivariant stable homotopy right inverse. 
\end{rem}

%

\subsection{Calculations via gluing theorem}  
In this subsection, we give several calculations of SWF homotopy contact invariants by the use of the gluing theorem. 
\begin{ex}
We consider the standard contact structure $\xi_{std}$ on $S^3$.  Our invariant lies in 
  \[
   \Psi (S^3, \xi_{std} ) \in \pi_{0}^S \cong \Z. 
   \]
Since $(S^3, \xi_{std})$ has a standard symplectic filling $(D^4, \om_{std})$, we have 
\[
\eta \circ (\Psi(D^4, \s_{\om_{std}}): S^0 \to S^0) \wedge  (\Psi (S^3, \xi_{std} ): S^0 \to S^0)  
\]
\[
= \Psi (D^4, \s_{\om_{std}},  \xi_{std} ): S^0 \to S^0. 
\]
Since $\Psi (D^4, \s_{\om_{std}},  \xi_{std} ): S^0 \to S^0$ is a $\pm 1$ map by \cite{I19}, we conclude that 
  \[
   \Psi (S^3, \xi_{std} ) \in \pi_{0}^S \cong \Z. 
   \]
   is a generator. 
  \end{ex}
We also give several calculations of our invariants for $\Sigma(2,3,r)$. The following calculations of Seiberg-Witten Floer homotopy types were given in \cite{Man03}, \cite{Ma14} using the result of \cite{MOY97}. The $d_3$-invariants can be computed from the results \cite{OS03}, \cite{GV16} and \cite{Bu20}. Here we use a relation between the $\Q$-grading of Heegaard Floer homology and $d_3$ given in \cite[Proposition~ 4.6]{OS05}. 

\begin{table}[htb]
\tiny
  \begin{tabular}{|l|c|r|r|r|} \hline
 & $SWF(Y, \s)$ & non-equivariant  & $d_3(Y)$ & $\Sigma^{-\frac{1}{2}-d_3(Y)}SWF(Y, \s)$    \\ \hline 
             $\Sigma(2,3, 12n+5)$&  $\Sigma^{\frac{1}{2} \mathbb{H} } (S^0 \vee \vee_n \Sigma^{-1} G_+) $ & $S^2 \vee \vee_{2n} (S^2 \vee S^1)$   &$\frac{3}{2} $  &  $S^0 \vee \vee_{2n} (S^0 \vee S^{-1})$  \\ \hline
$\Sigma(2,3, 12n-1)$  & $ \wt{G} \vee \vee_{n-1} \Sigma G_+$ &$S^2 \vee \vee_{2n-1} (S^2 \vee S^1)$  & $\frac{3}{2} $ &$S^0 \vee \vee_{2n-1} (S^0 \vee S^{-1})$   \\ \hline
     $\Sigma(2,3, 12n-5)$ & $ \Sigma^{-\frac{1}{2} \mathbb{H} } ( \wt{G} \vee \vee_{n-1} \Sigma G_+ )$  & $S^0 \vee \vee_{2n} (S^0 \vee S^{-1})$ & $-\frac{1}{2} $  & $S^0 \vee \vee_{2n} (S^0 \vee S^{-1})$  \\ \hline
        $\Sigma(2,3, 12n+1)$& $ S^0 \vee \vee_{n} \Sigma^{-1} G_+$ & $S^0 \vee \vee_{2n} (S^0 \vee S^{-1})$&  $-\frac{1}{2} $  &$S^0 \vee \vee_{2n} (S^0 \vee S^{-1})$   \\ \hline
             $-\Sigma(2,3, 12n+5)$ &$\Sigma^{-\frac{1}{2} \mathbb{H} } (S^0 \vee \vee_n  G_+) $  &$S^{-2} \vee \vee_{2n} (S^{-2} \vee S^{-1})$   & $-\frac{3}{2} $ &$S^{-1} \vee \vee_{2n} (S^0 \vee S^{-1})$   \\ \hline
       $-\Sigma(2,3, 12n-1)$  &$\Sigma^{-\mathbb{H}} ( \wt{T} \vee \vee_{n-1} \Sigma^2 G_+ )$  &$S^{-2} \vee \vee_{2n-1} (S^{-2} \vee S^{-1})$   &  $-\frac{3}{2} $ &$S^{-1} \vee \vee_{2n-1} (S^0 \vee S^{-1})$ \\ \hline
     $-\Sigma(2,3, 12n-5)$ & $\Sigma^{-\frac{1}{2} \mathbb{H} } ( \wt{T} \vee \vee_{n-1} \Sigma^2 G_+ )$  &$S^0 \vee \vee_{2n} (S^0 \vee S^{1})$  &$\frac{1}{2} $  & $S^{-1} \vee \vee_{2n} (S^0 \vee S^{-1})$ \\ \hline
    $-\Sigma(2,3, 12n+1)$ & $ S^0 \vee \vee_{n} G_+$  & $S^0 \vee \vee_{2n} (S^0 \vee S^{1})$ &$\frac{1}{2} $ & $S^{-1} \vee \vee_{2n} (S^0 \vee S^{-1})$ \\ \hline
  \end{tabular}
\end{table}
Here $G_+$ , $\wt{G}$ and $\wt{T} $ are the same notation given in \cite{Ma14}. We remark that the value of $d_3$ in the table is only for contact structures which can have a symplectic filling. 
\begin{ex}
For example, we can detect the dual of our invariant for the fillable contact structure $\xi_{std}$ on $Y=\Sigma(2,3,5)$ as a homotopy equivalence
\[
\widecheck{\Psi} (Y, \xi) :  \Sigma^{-\frac{1}{2}- d_3(Y)}SWF(Y, \s)=S^0  \to S^0, 
\]
where $\widecheck{\Psi} (Y, \xi)$ is the dual map of ${\Psi} (Y, \xi)$ introduced in the end of \cref{MMMM}. Moreover, we can similarly determine our invariant for a fillable contact structure of $-\Sigma(2,3,11)$.  
\end{ex}

\section{Applications to symplectic fillings}\label{Applications to symplectic fillings} 
In this section,  by the use of the gluing theorem and K or KO theory, we give several constraints of spin symplectic fillings. Moreover, at the end of this section, we also treat the extension property of a positive scalar curvature metric on a $4$-manifold with boundary.   

\par
Seiberg-Witten Floer homotopy type is a "formal desuspension" of a certain space, so in order to apply KO theory, we have to do some suspension in order to cancel the "formal desuspension".
More explicitly, for a $Spin^c$ cobordism $(W, \s_W): (Y, \s)\to (Y', \s')$ with $b_1(W)=b_1(Y)=b_1(Y')=0$, the relative Bauer-Furuta invariant is a morphism
\[
BF(W, \s_W): \Sigma^{\R^{-b^+(W)}\oplus \C^{\frac{c^2_1(\s_W)-\sigma(W)}{8}}}SWF(Y, \s)\to SWF(Y', \s')
\]
and in order to apply $KO$-theory, we have to choose $(m, n) \in \Z\times \Q$ large enough such that $n+\frac{c^2_1(\s_W)-\sigma(W)}{8}$ is an integer.
Then, we obtain an $S^1$ equivariant continuous map
\[
\Sigma^{\R^{m-b^+(W)}\oplus \C^{n+\frac{c^2_1(\s_W)-\sigma(W)}{8}}}SWF(Y, \s)\to SWF(Y', \s'). 
\]
Similarly, when $\s_W$ is spin, 
the relative Bauer-Furuta invariant is a morphism
\[
BF(W, \s_W): \Sigma^{\tilde{\R}^{-b^+(W)}\oplus \H^{-\sigma(W)/16}}SWF(Y, \s)\to SWF(Y', \s')
\]
and in order to apply $KO$-theory, we have to choose $(m, n) \in \Z\times \Q$ large enough such that $n-\frac{\sigma(W)}{16}$ is an integer and obtain a $Pin(2)$ equivariant continuous map  
\[
\Sigma^{\tilde{\R}^{m-b^+(W)}\oplus \H^{n-\sigma(W)/16}}SWF(Y, \s)\to \Sigma^{\tilde{\R}^{m}\oplus \H^{n}}SWF(Y', \s').
\]

\par
For example, 
$\R P^3=L(2, 1)=S^3_{-2}(\text{unknot})$ oriented as quotient $S^3/\Z_2$ has two isomorphism classes of $Spin^c$ structures $\s_0, \s_1$ , where $\s_0, \s_1$ are determined as follows:
On the cobordism   $W: S^3\to \R P^3$ obtained by a 2-handle attachment,  
we have $Spin^c$ structures $\hat{\s}_0, \hat{\s}_1$ such that
\[
\langle c_1(\hat{\s}_i), h \rangle =2i-2\quad i=0, 1.
\]
$\s_0, \s_1$ are the restriction of $\hat{\s}_0, \hat{\s}_1$.
In particular, 
\[
c_1(\hat{\s}_1)=0
\]
and thus $\hat{\s}_1,  \s_1$ is spin. We can check that $\s_0$ is also spin.
The intersection form of $W$ is $[-2]$ and thus its signature is $\sigma(W)=-1$.

It is proved in \cite{Man03} and \cite{Man07} that the Seiberg-Witten Floer homotopy type of $(\R P^3, \s_0)$ is given by
\[
SWF(\R P^3, \s_0)=(\C^{-1/8})^+=(\H^{-1/16})^+.
\]
The Seiberg-Witten Floer homotopy type of $(\R P^3, \s_1)$ is given by
\[
SWF(\R P^3, \s_1)=(\C^{1/8})^+=(\H^{1/16})^+. 
\]

The relative Bauer-Furuta invariant of $(W, \hat{\s}_0)$ is
\[
 BF(W, \hat{\s}_0): (\R^{-b^+(W)}\oplus \C^{\frac{c^2_1( \hat{\s}_0)-\sigma(W)}{8}})^+= (\R^{-1}\oplus \C^{-1/8})^+
 \]
 \[
 \to SWF(\R P^3, \s_0)=(\C^{-1/8})^+.
\]
Here we use the fact that
\[
c^2_1( \hat{\s}_i)=-\frac{(2i-2)^2}{2}.
\]
In order to take $KO$-theory, we have to choose $(m, n) \in \Z\times \Q$ large enough such that $n-\frac{1}{8}$ is an integer and obtain a continuous map
\[
(\R^{m-1}\oplus \C^{n-1/8})^+\to (\R^{m}\oplus \C^{n-1/8})^+
\]
\par
$(W, \hat{\s}_1)$ gives an example of spin cases.
The relative Bauer-Furuta invariant of $(W, \hat{\s}_1)$ is
\[
 BF(W, \hat{\s}_1): (\tilde{\R}^{-b^+(W)}\oplus \H^{-\sigma(W)/16})^+= (\tilde{\R}^{-1}\oplus \H^{1/16})^+
 \]
 \[
 \to SWF(\R P^3, \s_1)=(\H^{1/16})^+
\]
and 
in order to take $Pin(2)$ equivariant $KO$-theory, we have to choose $(m, n) \in \Z\times \Q$ large enough such that $n+\frac{1}{16}$ is an integer.

\subsection{Two constraints for symplectic fillings of homotopy L-spaces}
In this subsection, we will give a constraint for symplectic fillings. 
However, all results in this subsection also can be proved by using monopole Floer homology or Heegaard Floer homology. 
In order to introduce our theorem in this section, we introduce a notion of (Seiberg-Witten) homotopy L-spaces. 
\begin{defn}
A $Spin^c$ rational homology $3$-sphere $(Y, \s)$ is a {\it homotopy L-space} if  
\begin{align}\label{L-sp}
SWF(Y, \s) = \C^\delta
\end{align}
for some rational number $\delta$.
A rational homology $3$-sphere $Y$ is a {\it homotopy L-space} if, for any $Spin^c$ structures on $Y$, $(Y, \s)$ is a homotopy L-space. 
\end{defn} 
Note that $\delta$ coincides with the Fr\o yshov invariant $\delta(Y, \s)$. 
We compare homotopy L-spaces with L-spaces. Usually, L-space is defined using the Heegaard Floer homology. For example, one definition is $\widehat{HF}(Y)$ is free and \[
\rank (\widehat{HF}(Y):= \bigoplus_{\s }\widehat{HF}(Y, \s)= |H^2(Y; \Z) | .
\] 
In the work of Kutluhan, Lee, and Taubes \cite{KLTI}, \cite{KLTII}, \cite{KLTIII}, \cite{KLTIV}, \cite{KLTV} , alternatively, the work of Colin, Ghiggini, and Honda \cite{CGHI} \cite{CGHII} \cite{CGHIII}  and Taubes \cite{Ta10}, \cite{TII10}, \cite{TIII10}, \cite{TIV10}, \cite{TV10}, it is proved that 
\begin{align}\label{isom1}
\widehat{HF}_*(Y)    \cong   \wt{HM}_*(Y) :=  \bigoplus_{\s }\wt{HM}_*(Y, \s). 
\end{align}
 Note that Lidman-Manolescu's(\cite[Corollary 1.2.2]{LM18}) constructed an isomorphism 
 \begin{align}\label{isom2}
\wt{HM}_*(Y, \s) \cong \wt{H}_* (SWF(Y, \s)). 
\end{align}
 By combining these two isomorphisms, one can confirm that any homotopy L-space is an L-space. 
 \begin{ques} Is there an L-space which is not a homotopy L-space? 
 \end{ques}
 However, the authors do not know whether the converse is true or not. Of course, spherical 3-manifolds are homotopy L-spaces. Moreover, F.Lin and Lipnowski (\cite{LL18}) provided hyperbolic examples of homotopy L-spaces.

The proofs of \cref{KO} are similar to the proof of \cref{non-spin}. 
\begin{thm}\label{non-spin}Let $(Y, \xi)$ be a contact rational homology $3$-sphere. Suppose that $(Y,\s_\xi)$ is a homotopy L-space.  
Then, for any symplectic filling $(X, \om)$ of $(Y, \xi)$, the following two facts hold: 
\begin{itemize}
\item[(i)] $b^+(X) =0$ and 
\item[(ii)] $\displaystyle \frac{c_1(\s_\om)^2+ b_2(X)}{8} = \delta (Y, \s_\xi) $. 
\end{itemize}
\end{thm}
Note that the second equality can be regarded as the "opposite direction" of Fr\o yshov's inequality(\cite{Fr10}), which is a generalization of Donaldson's diagonalization theorem to negative definite 4-manifolds with boundary. Philosophically, F. Lin's result(\cite{L20}) can be seen as constraints corresponding to Donaldson's Theorem B and C.

The result (i) was proved in the case of 3-manifolds admitting a positive scalar curvature metric \cite{Li98}, L-spaces \cite{OS04} and \cite{E20}. The result (ii) follows from the fact that Kronheimer-Mrowka-Ozsv\'ath-Szab\'o' s contact invariant $\psi(Y, \xi)$ is contained in the kernel of the $U$-map and non-zero for strongly fillable contact structure, so it belongs to the bottom of the $U$-tower. 
We give a homotopy theoretic proof of \cref{non-spin}. 
\begin{proof}[Proof of \cref{non-spin}] 
As in the proof of \cref{right}, we obtain the following diagram
\begin{align}\label{zusiki}
 \begin{CD}
 S^{m  + 2n +\frac{c^2_1(\s_X)-\sigma(X)}{4}} @>{\Psi_{(X, \s_X)}}>>   \Sigma^{m+b^+(X) + 2n} SWF(Y)    \\
  @.    @V{\Psi(Y, \xi)\wedge id }VV    \\
  @.  \Sigma^{\frac{1}{2} - d_3(-Y, [\xi]) } SWF(-Y) \wedge  \Sigma^{m+b^+(X) + 2n} SWF(Y)  \\
  @. @\vert \\ 
@.   \Sigma^{\frac{1}{2} - d_3(-Y, [\xi])  + m + b^+(X) + 2n } SWF(-Y)  \wedge SWF(Y)  \\
@.    @V{ id \wedge \eta  }VV    \\
@. S^{\frac{1}{2} - d_3(-Y, [\xi])  + m + b^+(X) + 2n} \\
 \end{CD}, 
\end{align}
where $(m,n ) \in \Z \times \Q$ such that $n +\frac{c^2_1(\s_X)-\sigma(X)}{8} \in \Z$.
This diagram commutes up to stable homotopy. The gluing theorem implies that 
$(id \wedge \eta) \circ  (\Psi(Y, \xi)\wedge id) \circ  {\Psi_{(X, \s_X)}}$ and $\Psi_{(X, \s_{\om}, \xi)}$ are stably homotopic. 
Since $(X, \om)$ is a symplectic filling, $\Psi_{(X, \s_{\om}, \xi)} $ is a homotopy equivalence. 
This implies 
\[
\frac{c^2_1(\s_X)-\sigma(X)}{4} = b^+(X) + {2} \delta (Y, \s_\xi) . 
\]
Moreover, the mapping degree of $\Psi_{(X, \s_X)} $ is $\pm 1$.
Apply \cite[Proof of 1.3]{BF04} to the $S^1$-equivariant Bauer-Furuta invariant 
\[
\Psi_{(X, \s_X)} : ( \R^m \oplus \C^{\frac{c^2_1(\s_X)-\sigma(X)}{8} +n})^+  \to   (\R^{m+b^+(X)} +  \C^{n+\delta} )^+ ,
\]
we have $b^+(X)=0$. Thus, we have $\displaystyle \frac{c_1(\s_\om)^2+ b_2(X)}{8} = \delta (Y, \s_\xi) $. 
\end{proof}
\subsection{Equivariant KO theory} 
In this subsection, we will use KO theory. We treat a class of symplectic fillings $(X, \om)$ whose $\s_\om$ are spin. In the same spirit of (i) in \cref{non-spin}, we give upper bounds of $ b^+$ for spin symplectic fillings.

In particular, we give a proof of \cref{KO}.  
We will use the relative Bauer-Furuta invariant of "upside-down" $X_\dagger$ of $X$  in this section, which is a morphism
\[
\Psi(X_\dagger, \s_\omega): SWF(-Y, \s_\xi)\to (\tilde{\R}^{b^+(X)}\oplus \mathbb{H}^{\frac{\sigma(X)}{16}})^+.
\]
and in order to apply $KO$-theory, we need to take a suspension
\[
\Psi(X_\dagger, \s_\omega): \Sigma^{\tilde{\R}^m\oplus \mathbb{H}^n}SWF(-Y, \s_\xi)\to (\tilde{\R}^{b^+(X)+m}\oplus \mathbb{H}^{\frac{\sigma(X)}{16}+n})^+.
\]
such that $n+\frac{\sigma(X)}{16} \in \Z$.
Notice that the sign $+\frac{\sigma(X)}{16}$ is different from the one we mentioned at the beginning of this section.

\subsubsection{Proof of \cref{KO}}
We focus on the proof of \cref{KO}. 
We write $G=Pin(2)$ in this section.
The following periodicity is known: 
\begin{lem}\cite[Section 2.2]{Lin15}Let $k$ and $l$ be non-negative integers. Then we have isomorphisms
\[
\wt{KO}_G((\tilde{\R}^{k}\oplus \mathbb{H}^{l})^+)
\cong \wt{KO}_G((\tilde{\R}^{k+8}\oplus \mathbb{H}^{l})^+)
\]
\[
\cong\wt{KO}_G((\tilde{\R}^{k+4}\oplus \mathbb{H}^{l+1})^+)\cong \wt{KO}_G((\tilde{\R}^{k}\oplus \mathbb{H}^{l+2})^+). 
\]
 \end{lem}
 Let $(Y, \s)$ be a spin rational homology $3$-sphere. 
We consider an equivalence relation $\sim_{KO}$ on 
\[
 \Z \times \left\{\  l \in  \frac{1}{16} \Z\  \middle| l\  +  \frac{\sigma (X ) } {16} \in \Z  \ \right\} 
 \]
  by the following way; 
\begin{itemize}
\item 
$(k , l ) \sim_{KO} (k+8, l ) $, 
\item $(k , l ) \sim_{KO} (k+4, l+1 ) $, and
\item $(k , l ) \sim_{KO} (k, l+2 ) $, 
\end{itemize}
where $\sigma (X)$ is the signature of a spin bounding of $(Y, \s)$. 
The notion $J_{KO}(Y, \s ) $ denotes the quotient set $ \Z \times \left\{ l \in  \frac{1}{16} \Z \middle| l +  \frac{\sigma (X ) } {16} \in \Z  \right\} $ divided by $\sim_{KO}$. 
We consider representatives of $J_{KO}(Y, \s) $ as 
\[
\left\{ [(0, l_0)], [(1, l_1)] , [(2, l_2 )],  [(3, l_3 )]  \middle|  l_i \in \left\{0,  \frac{1}{16}, \cdots, \frac{31}{16} \right\} ,~  l_i + \frac{\sigma (X ) } {16} \in \Z   \right\}. 
\]
 
\begin{defn}
For a rational homology $3$-sphere $Y$ with a spin structure $\s$ and $[(m, n)] \in J_{KO}$ with $n + \frac{\sigma(X)}{16} \in \Z$, we have two groups 
\[
KOM^{-m,-n}_{G}  (Y, \s ) :=   \wt{KO}_{G} (\Sigma^{  m\wt{\R} \oplus n {\H}   } SWF(Y, \s)  ) 
\]
and its reducible part 
\[
\overline{KOM}^{-m}_{G} (Y, \s) :=   \wt{KO}_{G} ((\Sigma^{ m \wt{\R} } SWF(Y, \s) )^{S^1} ) . 
\]
We call $ KOM^{-m,-n}_{G}  (Y, \s )$ {\it Seiberg-Witten Floer KO-homology}. 
\end{defn}
We associate a homomorphism 
\[
i^*_{m,n}:KOM^{-m,-n}_{G}  (-Y, \s_\xi ) \to \overline{KOM}^{-m}_{G} (-Y, \s_\xi )
\]
and 
\[
\varphi_{m} : \overline{KOM}^{-m}_{G} (-Y, \s_\xi ) \to \Z
\]
where $i$ is the inclusion map $(\Sigma^{m\wt{\R} } SWF(-Y))^{S^1} \to \Sigma^{m\wt{\R} \oplus n {\H}  } SWF(-Y)$ and the map $\varphi_{m}$ is introduced by Jianfeng Lin \cite[Definition~5.1]{Lin15}. 
\par
In this section, we prove the following theorem stated in the introduction: 
\begin{thm}Let $(Y, \s)$ be a spin rational homology 3-sphere and $(m,n )$ be a representative of an element in $J_{KO} (Y, \s) $. 
When 
\[
-d_3(Y, [\xi])-\frac{1}{2}+m+4n\equiv 0,  4\quad\text{ mod }8, 
\] 
suppose also that the following induced map from $\varphi_{m} \circ i^*_{m,n}$
\[
 (KOM^{-m,-n}_{G}  (-Y, \s_\xi )/\operatorname{Torsion}) \otimes \Z_2 
  \to  \Z_2
\]
is injective. 
When 
\[
-d_3(Y, [\xi])-\frac{1}{2}+m+4n\equiv 1,  2\quad\text{ mod }8, 
\] 
suppose also that the following induced map from $\varphi_{m} \circ i^*_{m,n}$
\[
 KOM^{-m,-n}_{G}  (-Y, \s_\xi ) \otimes \Z_2 
  \to  \Z_2
\]
is injective. 

Then any symplectic filling $(X, \om)$ of $(Y, \xi)$ satisfying $ \s_{\om}$ is spin and $b_1(X)=0$ satisfies   
\[
b^+ (X) \leq \mathfrak{e} (m), 
\]
where 
\[
 \mathfrak{e} (m)=
 \begin{cases}
 0&m\equiv 0, 1, 2, 4\text{ mod } 8\\
 1&m\equiv 3, 7\text{ mod } 8\\
 2&m\equiv 6\text{ mod } 8\\
 3&m\equiv 5\text{ mod } 8.
 \end{cases}
\]
In particular, 
\[
b^+ (X) \leq 3. 
\]
\end{thm}
\begin{proof}
Note that for any spin filling $(X, \om)$ with $b_1(X)=0$, 
\[
b^+(X)+m+\frac{\sigma(X)}{4}+4n=-d_3(Y, [\xi])-\frac{1}{2}+m+4n
\]
by the definition of $d_3$. Let $m$ be a sufficiently large integer and $n$ be a sufficiently large rational number such that $n+\frac{\sigma(X)}{16}$ is an integer. Denote by $X_\dagger$ the '' upside-down'' cobordism $-Y\to \emptyset$ obtained from $X$, which is the same as $X$ as an oriented manifold (not orientation reversed one), and consider its relative Bauer-Furuta invariant
\[
\Psi(X_\dagger, \s_\omega): \Sigma^{\tilde{\R}^m\oplus \mathbb{H}^n}SWF(-Y, \s_\xi)\to (\tilde{\R}^{b^+(X)+m}\oplus \mathbb{H}^{\frac{\sigma(X)}{16}+n})^+.
\]
We denote our contact invariant by
\[
\Psi(Y, \xi): (\R^{b^+(X)+m+\frac{\sigma(X)}{4}+4n})^+\to \Sigma^{\tilde{\R}^m\oplus \mathbb{H}^n}SWF(-Y, \s_\xi), 
\]
which is a non-equivariant map.
Consider the following commutative diagram:
\begin{align}\label{comm}
 \begin{CD}
  @. \wt{KO}((\R^{b^+(X)+m+\frac{\sigma(X)}{4}+4n})^+) \\ 
@. @A{\Psi(Y, \xi)^*}AA    \\
\wt{KO}((\tilde{\R}^{b^+(X)+m}\oplus \mathbb{H}^{\frac{\sigma(X)}{16}+n})^+)  @>{\Psi(X_\dagger, \s_\omega)^*}>>  \wt{KO}(\Sigma^{\tilde{\R}^m\oplus \mathbb{H}^n}SWF(-Y, \s_\xi)) \\ 
@A{r'}AA  @A{r}AA    \\
\wt{KO}_G((\tilde{\R}^{b^+(X)+m}\oplus \mathbb{H}^{\frac{\sigma(X)}{16}+n})^+)  @>{\Psi(X_\dagger, \s_\omega)^*}>>  \wt{KO}_G(\Sigma^{\tilde{\R}^m\oplus \mathbb{H}^n}SWF(-Y, \s_\xi)) \\ 
  @V{} VV  @V{i^*_{m, n}}VV    \\
\wt{KO}_G((\tilde{\R}^{b^+(X)+m})^+)@>{ (\Psi(X_\dagger, \s_\omega)^{S^1})^* }>>   \widetilde{KO}_{G} ((\tilde{\R}^m)^+ )  \\ 
   @V{ \varphi_{b^+(X)+m}} VV  @V{\varphi_{m}}VV    \\
     \Z @>{ 2^{\alpha_{m+1}+\cdots+ \alpha_{m+b^+(X)}}   }>>   \Z  \\ 
 \end{CD}, 
\end{align}
where $r$ and $r'$ are the forgetful maps, $\varphi_k$ is defined in \cite[Definition 5.1]{Lin15} and
\[
\alpha_i=
\begin{cases}
1&i\equiv 1, 2, 3, 5\quad \text{ mod }8\\
0&\text{otherwise}
\end{cases}
\]
as in \cite[Definition 5.2]{Lin15}.
\par
When
\[
b^+(X)+m+\frac{\sigma(X)}{4}+4n=-d_3(Y, [\xi])-\frac{1}{2}+m+4n\equiv 0, 1, 2,  4\quad\text{mod }8, 
\] 
 the forgetful map
\[
r' :  \wt{KO}_G((\tilde{\R}^{b^+(X)+m}\oplus \mathbb{H}^{\frac{\sigma(X)}{16}+n})^+)\to \wt{KO}((\tilde{\R}^{b^+(X)+m}\oplus \mathbb{H}^{\frac{\sigma(X)}{16}+n})^+)
\]
can be regarded as
\[
\wt{KO}_G(S^0)\to \wt{KO}(S^0)\cong \Z\quad \text{when} -d_3(Y, [\xi])-\frac{1}{2}+m+4n\equiv 0\text{ mod }8,
\]
\[
\wt{KO}_G(\tilde{\R}^+)\to \wt{KO}(\tilde{\R}^+) \cong \Z/2\quad \text{when} -d_3(Y, [\xi])-\frac{1}{2}+m+4n\equiv 1\text{ mod }8,
\]
\[
\wt{KO}_G(\tilde{\R}^{2+})\to \wt{KO}(\tilde{\R}^{2+})\cong \Z/2\quad \text{when} -d_3(Y, [\xi])-\frac{1}{2}+m+4n\equiv 2\text{ mod }8,
\]
\[
\wt{KO}_G(\tilde{\R}^{4+})\to \wt{KO}(\tilde{\R}^{4+})\cong \Z\quad \text{when} -d_3(Y, [\xi])-\frac{1}{2}+m+4n\equiv 4\text{ mod }8,
\]
respectively via Bott periodicity
\[
\wt{KO}_G((\tilde{\R}^{k}\oplus \mathbb{H}^{l})^+)
\cong \wt{KO}_G((\tilde{\R}^{k+8}\oplus \mathbb{H}^{l})^+)
\]
\[
\cong\wt{KO}_G((\tilde{\R}^{k+4}\oplus \mathbb{H}^{l+1})^+)\cong \wt{KO}_G((\tilde{\R}^{k}\oplus \mathbb{H}^{l+2})^+).
\]
For $k\equiv 0, 1, 2, 4\text{ mod } 8$, fix a generator  $e_k \in \wt{KO}_G(\tilde{\R}^{k+})$ as follows:
\begin{itemize}
\item
In the case $k\equiv 0 \text{ mod } 8$,  $e_0$ corresponds to $1\in RO(G) \cong \wt{KO}_G(S^0)$.
\item
In the case $k\equiv 1 \text{ mod } 8$, 
$\wt{KO}_G(\tilde{\R}^+)\cong \Z$ and $e_1$ be either of the generators.
\item
In the case $k\equiv 2 \text{ mod } 8$, 
$\wt{KO}_G( (\tilde{\R}^{2})^+ )\cong \Z\oplus \oplus_{m \geq 0}\Z/2$
and the generators are $\eta(D)^2$ and $\gamma(D)^2A^m c$, where the notation is explained in \cite[Proposition 5.5]{Sc03}.
The element $e_3$ is the generator corresponding to $\eta(D)^2$.
\item
In the case $k\equiv 4 \text{ mod } 8$, 
$\wt{KO}_G((\tilde{\R}^{4})^+)$ is freely generated by 
\[
\lambda(D), D\lambda(D), A^n \lambda(D)\text{ and }A^mc, 
\]
where the notation is explained in \cite[Proposition 5.5]{Sc03}.
The element $e_4$ is the generator corresponding to $\lambda(D)$.
\end{itemize}
For the above description of $\wt{KO}_G((\tilde{\R}^{k})^+)$, see  \cite[Proposition 5.5]{Sc03} and \cite[Theorem 2.13]{Lin15} ).
We can check that the image of $e_k$ under the forgetful map is a generator of $\wt{KO}((\tilde{\R}^{k})^+)$.
In each case of $k$, we set 
\[
x:=\Psi(X_\dagger, \s_\omega)^*e_k \in \wt{KO}_G(\Sigma^{\tilde{\R}^m\oplus \mathbb{H}^n}SWF(-Y, \s_\xi)).
\]
 \cref{gluing theorem} implies that the composition
\[
 (\R^{b^+(X)+m+\frac{\sigma(X)}{4}+4n})^+\xrightarrow{\Psi(Y, \xi)}
    \Sigma^{\tilde{\R}^m\oplus \mathbb{H}^n}SWF(-Y, \s_\xi)
    \]
    \[
    \xrightarrow{\Psi(X_\dagger, \s_\omega)}(\tilde{\R}^{b^+(X)+m}\oplus \mathbb{H}^{\frac{\sigma(X)}{16}+n})^+
\]
is homotopic to the Bauer-Furuta version of Kronheimer-Mrowka's invariant of $(X, \s_\omega)$.
The facts that the mapping degree of the Bauer-Furuta version $\Psi(X, \xi , \s_\om )$ of Kronheimer-Mrowka's invariant equals Kronheimer-Mrowka's invariant up to sign and the non-vanishing theorem of Kronheimer-Mrowka's invariant for weak symplectic fillings (Theorem1.1 in \cite{KM97}), which imply that this map is a homotopy equivalence.
Thus, the composition
\[ 
\wt{KO}((\tilde{\R}^{b^+(X)+m}\oplus \mathbb{H}^{\frac{\sigma(X)}{16}+n})^+)  \xrightarrow{\Psi(X_\dagger, \s_\omega)^*}
\]
\[
  \wt{KO}(\Sigma^{\tilde{\R}^m\oplus \mathbb{H}^n}SWF(-Y, \s_\xi)) \xrightarrow{\Psi(Y, \xi)^*} \wt{KO}((\R^{b^+(X)+m+\frac{\sigma(X)}{4}+4n})^+) 
\]
is an isomorphism, so the image of $r'(e_k)$ under this map is \[
\pm 1 \in \wt{KO}((\R^{b^+(X)+m+\frac{\sigma(X)}{4}+4n})^+) \cong \Z \text{ or } \Z/2.
\]
\begin{itemize}
\item[(i)] 
When
\[
b^+(X)+m+\frac{\sigma(X)}{4}+4n=-d_3(Y, [\xi])-\frac{1}{2}+m+4n\equiv 0,  4\quad\text{ mod }8, 
\] 
commutativity of the diagram implies that
\[
x\neq 0  \in (\wt{KO}_G(\Sigma^{\tilde{\R}^m\oplus \mathbb{H}^n}SWF(-Y, \s_\xi)) /\operatorname{Torsion}) \otimes \Z/2 .
\]
Indeed, if $x$ were written as $x= 2x'+(\text{torsion})$ for some $x' \in \wt{KO}_G(\Sigma^{\tilde{\R}^m\oplus \mathbb{H}^n}SWF(-Y, \s_\xi))$, 
\[
\pm 1=\Psi(Y, \xi)^*\circ r(x)=2\Psi(Y, \xi)^*\circ r(x') \in \wt{KO}((\R^{b^+(X)+m+\frac{\sigma(X)}{4}+4n})^+)\cong \Z, 
\]
which is a contradiction.
\item[(ii)] 
When 
\[
b^+(X)+m+\frac{\sigma(X)}{4}+4n=-d_3(Y, [\xi])-\frac{1}{2}+m+4n\equiv 1,  2\quad\text{ mod }8, 
\] 
commutativity of the diagram implies that 
\[
x\neq 0 \in \wt{KO}(\Sigma^{\tilde{\R}^m\oplus \mathbb{H}^n}SWF(-Y, \s_\xi)) \otimes \Z/2 .
\]
Indeed, if $x$ were written as $x= 2x'$ for some $x' \in \wt{KO}_G(\Sigma^{\tilde{\R}^m\oplus \mathbb{H}^n}SWF(-Y, \s_\xi))$, 
\[
\pm 1=\Psi(Y, \xi)^*\circ r(x)=2\Psi(Y, \xi)^*\circ r(x') =0 \in \wt{KO}((\R^{b^+(X)+m+\frac{\sigma(X)}{4}+4n})^+)\cong \Z/2, 
\]
which is a contradiction.
\end{itemize}
The injectivity hypothesis of the theorem implies that $\varphi_m \circ i^*_{m, n}(x) \in \Z$ is not even.
\par
Now, suppose to the contrary that $b^+(X)>\mathfrak{e}(m)$.
Since
\[
\mathfrak{e}(m)=\min\{b \in \Z^{\geq 1} |\alpha_{m+1}+\cdots+\alpha_{m+b}\geq 1\}-1, 
\]
we have
\[
\alpha_{m+1}+\cdots+\alpha_{m+b^+(X)}\geq 1.
\]
and thus $2^{\alpha_{m+1}+\cdots+\alpha_{m+b^+(X)}}$ is even.
Commutativity of the lower part of the diagram implies $\varphi_m \circ i^*_{m, n}(x)$ is even, contradicting the above argument.
\end{proof}
\subsubsection{Examples}
The following result is contained in the F.Lin's argument of \cite{L20}, but we give an alternative
proof using \cref{KO}.
\begin{prop}
Let $(X, \omega)$ be a symplectic filling of some contact structure of $-\Sigma(2, 3, 11)$ such that $\s_\omega$ is spin and $b_1(X)=0$. 
Then $b^+(X)=1$.
\end{prop} 
\begin{proof}
In \cite{GV16}, it is showed that every tight (in particular fillable) contact structure $\xi$ on
$-\Sigma(2, 3, 11)$ have
\[
d_3(-\Sigma(2, 3, 11), \xi)=-\frac{3}{2}.
\]
Since 
\[
d_3(-\Sigma(2, 3, 11), \xi)=-\frac{\sigma(X)}{4}-b^+(X)-\frac{1}{2}
\]
and Rokhlin invariant of $-\Sigma(2, 3, 11)$ is zero, 
\[
b^+(X)=-\frac{\sigma(X)}{4}+1
\]
must be odd.
Thus,  it is enough to show $b^+(X)\leq 1$.
Manolescu showed in \cite{Man07} and \cite{Ma14} that
\[
SWF(\Sigma(2, 3, 11))=\tilde{G}, 
\]
i.e.  the unreduced suspension of $Pin(2)$.
Take sufficiently large $m\equiv -1\text{ mod } 8$, $n\equiv 0 \text{ mod } 2$, 
so that
\[
-d_3(-\Sigma(2, 3, 11), \xi)-\frac{1}{2}+m+4n\equiv 0\text{ mod } 8
\]
holds.
As in section 8.1 in  \cite{Lin15}, the exact sequence for pair for $(\Sigma^{\tilde{\R}^m}\tilde{G}, (\Sigma^{\tilde{\R}^m}\tilde{G})^{S^1})$ yields
\[
\cdots \to \wt{KO}_G((\tilde{\R}^{m+1})^+) \xrightarrow{A} \wt{KO}(S^{m+1})
\]
\[
 \to \wt{KO}_G(\Sigma^{\tilde{\R}^m}\tilde{G})\to \wt{KO}_G((\tilde{\R}^m)^+)\to \wt{KO}(S^{m})\to \cdots.
\]
Here the map $A$ can be regarded as the augmentation map $RO(G)\to \Z$, which is surjective.
Note also that $\wt{KO}_G((\tilde{\R}^m)^+)=0$ for $m\equiv -1\text{ mod } 8$.
Thus the exact sequence implies that $\wt{KO}_G(\Sigma^{\tilde{\R}^m}\tilde{G})\to \wt{KO}_G((\tilde{\R}^m)^+)$ is isomorphism and so is $i^*_{m, n}$.
Since the map
\[
\varphi_m: \wt{KO}_G((\tilde{\R}^m)^+)\to \Z
\]
 is given by the projection to the $\Z$-summand under the isomorphism
\[
\wt{KO}_G((\tilde{\R}^m)^+)\cong \Z\oplus \oplus_{n \geq 1}\Z/2
\]
as described in Theorem 2.13, Definition 5.1 in \cite{Lin15}, the hypothesis of the theorem is satisfied and we can conclude $b^+(X)\leq 1$. Here we use $m \equiv 7 \operatorname{mod} 8$.
\end{proof}


\bibliographystyle{plain}
\bibliography{tex}

\end{document}